\documentclass[a4paper, 11pt, twoside]{article}

\usepackage[top=3cm, bottom=3cm, left=3cm, right=3cm]{geometry}
\usepackage{amsmath}
\usepackage{amsthm}
\usepackage{amssymb}
\usepackage{mathtools}
\usepackage[integrals]{wasysym}
\usepackage{stmaryrd}
\usepackage[mathcal]{euscript}
\usepackage[utf8]{inputenc}
\usepackage[T1]{fontenc}
\usepackage[english]{babel}
\usepackage[noadjust]{cite}
\usepackage{lmodern}
\usepackage[inline]{enumitem}
\usepackage[hidelinks, bookmarks, bookmarksnumbered, pdfstartview={XYZ null null 1.00}]{hyperref}

\theoremstyle{plain}
\newtheorem{theorem}{Theorem}[section]
\newtheorem{lemma}[theorem]{Lemma}
\newtheorem{corollary}[theorem]{Corollary}
\newtheorem{proposition}[theorem]{Proposition}
\theoremstyle{definition}
\newtheorem{definition}[theorem]{Definition}
\newtheorem{remark}[theorem]{Remark}

\DeclareMathOperator{\Lip}{Lip}
\DeclareMathOperator{\diverg}{div}
\DeclareMathOperator{\supp}{supp}
\DeclareMathOperator{\dist}{{\mathbf{d}}}
\DeclareMathOperator{\MFG}{MFG}
\DeclareMathOperator{\OCP}{OCP}
\DeclareMathOperator{\Adm}{Adm}
\DeclareMathOperator{\Opt}{Opt}
\DeclareMathOperator{\OOpt}{{\mathbf{Opt}}}
\DeclareMathOperator{\Eq}{Eq}
\DeclareMathOperator{\EEq}{{\mathbf{Eq}}}
\DeclareMathOperator{\Lim}{Lim}

\DeclarePairedDelimiter{\abs}{\lvert}{\rvert}

\newcommand{\suchthat}{\ifnum\currentgrouptype=16 \mathrel{}\middle|\mathrel{}\else\mid\fi}
\newcommand{\diff}{\,\mathrm{d}}

\numberwithin{equation}{section}

\begin{document}

\setlength{\parskip}{1pt plus 1pt minus 1pt}

\setlist[enumerate, 1]{label={\textnormal{(\alph*)}}, ref={(\alph*)}, leftmargin=0pt, itemindent=*}
\setlist[itemize, 1]{label={\textbullet}, leftmargin=0pt, itemindent=*}

\newlist{hypothesis}{enumerate}{1}
\setlist[hypothesis]{label={\textup{(H\arabic*)}}, ref={(H\arabic*)}, leftmargin=*, widest*=10}

\def\paperTitle{{A note on existence and asymptotic behavior of Lagrangian equilibria for first-order optimal-exit mean field games}}
\newcommand{\paperAuthor}{Guilherme Mazanti\thanks{Université Paris-Saclay, CNRS, CentraleSupélec, Inria, Laboratoire des signaux et systèmes, 91190, Gif-sur-Yvette, France.}}
\newcommand{\paperKeywords}{Mean field games, Lagrangian equilibria, congestion games, existence of equilibria, asymptotic behavior}
\newcommand{\paperMSC}{49N80, 91A16, 93C15}

\expandafter\title\paperTitle
\author{\paperAuthor}

\date{}

\maketitle

\begin{abstract}
In this paper, we consider a first-order mean field game model motivated by crowd motion in which agents evolve in a (not necessarily compact) metric space and wish to reach a given target set. Each agent aims to minimize the sum of their travel time and an exit cost which depends on their exit position on the target set. Agents interact through their dynamics, the maximal speed of an agent being assumed to be a function of their position and the distribution of other agents. This interaction may model, in particular, congestion phenomena. Under suitable assumptions on the model, we prove existence of Lagrangian equilibria, analyze the asymptotic behavior for large time of the distribution of agents, and study the dependence of equilibria and asymptotic limits on the initial distribution of the agents.

\medskip

\noindent\textbf{2020 Mathematics Subject Classification.} \paperMSC{}.

\medskip

\noindent\textbf{Keywords.} \paperKeywords{}.
\end{abstract}

\hypersetup{pdftitle=\paperTitle, pdfauthor={\paperAuthor}, pdfkeywords={\paperKeywords}, pdfsubject={\paperMSC}}

\section{Introduction}

Since their introduction around 2006 by the simultaneous works of Peter E.~Caines, Minyi Huang, and Roland P.~Malhamé \cite{Huang2003Individual, Huang2007Large, Huang2006Large} and of Jean-Michel Lasry and Pierre-Louis Lions \cite{Lasry2007Mean, Lasry2006JeuxI, Lasry2006JeuxII}, following previous works in the economics literature on games with infinitely many agents \cite{Aumann1964Markets, Jovanovic1988Anonymous}, mean field games (MFGs) have attracted the interest of a large number of researchers, due both to the interesting and challenging theoretical questions raised by their analysis and their wide spectrum of applications, ranging from engineering and economics to the modeling of crowd motion and epidemics \cite{Lachapelle2011Mean, Aurell2022Optimal, Huang2003Individual, Lasry2007Mean}. Works on this topic have considered questions such as existence and uniqueness of their equilibria, approximation of games with many players by MFGs, numerical methods for approximating equilibria of MFGs, or the characterization of equilibria through the master equation, among others \cite{Lauriere2022Convergence, Carlini2014Fully, Achdou2010Mean, Porretta2018Turnpike, Gomes2021Numerical, Cardaliaguet2019Master}. We refer to \cite{Carmona2018ProbabilisticI, Carmona2018ProbabilisticII, Achdou2020MeanI} for more details on recent topics on mean field games.

While the majority of the works on mean field games consider that the agents of the game evolve in a given time interval $[0, T]$, many applications involve agents that may leave the game before its terminal time $T$, and also games in which a terminal time $T$ is not prescribed and agents evolve in $[0, +\infty)$, leaving the game at some point. This is the case, for instance, of \cite{Carmona2017Mean, Burzoni2023Mean}, which consider mean field games with applications to bank run, i.e., situations in which clients of a bank, believing that the bank is about to fail, withdraw all their money, and try to choose the time to withdraw the money in an optimal way. These models belong to a more general class of mean field game problems known as mean field games of optimal stopping, in which the main choice of an agent is when to stop the game \cite{Bertucci2018Optimal, Gomes2015Obstacle, Nutz2018Mean, Bouveret2020Mean}. Other works, such as \cite{Graber2020Mean}, consider economic models for the production of exhaustible resources, in which firms, who wish to maximize their profit, produce goods based on exhaustible resources, and they leave the game when they deplete their capacities. These games can be seen as mean field games with an absorbing boundary, i.e., in which agents who reach a certain part of the boundary of the domain immediately leave the game \cite{Campi2018NPlayer}. An important kind of optimal control problem with an absorbing boundary is that of conditional exit control problems, studied in details in \cite{CarmonaNonstandard}, in which the running cost of an agent at time $t$ is conditioned to the fact that the agent is still in the domain at time $t$. In the sequel, we refer to MFGs in which agents choose the time at which they leave the game as \emph{optimal-exit} MFGs.

This paper considers a first-order optimal-exit MFG in which a continuum of rational agents evolve in a given metric space $X$, the aim of each agent being to reach a given target set $\Gamma \subset X$ while minimizing the sum of the time they take to reach $\Gamma$ and a cost depending on the arrival position at $\Gamma$. We also assume that, at each time, the velocity of an agent is bounded by a function depending on their position and the current distribution of other agents, that is,
\begin{equation}
\label{eq:intro-velocity-bound}
\abs{\dot\gamma}(t) \leq K(m_t, \gamma(t)),
\end{equation}
where $\gamma$ is the trajectory of the agent, $\abs{\dot\gamma}$ is its metric derivative, $m_t$ is a probability measure on $X$ describing the distribution of agents at time $t$, and $K$ is a function taking positive values. This MFG model is inspired by the study of crowd motion: agents of the game may represent pedestrians moving in the spatial domain $X$ who wish to reach an exit of $X$, the set of possible exits of $X$ being $\Gamma$. In this context, the bound on the velocity of an agent from \eqref{eq:intro-velocity-bound} can be interpreted as a model for congestion, which, roughly speaking, should model the fact that an agent is physically blocked by other agents in regions of large density, and thus they cannot move faster than a certain maximal speed, which depends on their position and the density of other agents around their position.

Motivated by understanding and, if possible, controlling and optimizing the flow of large groups of people, several works have addressed the mathematical modeling of crowd motion \cite{Helbing1995Social, Henderson1971Statistics, Maury2019Crowds}. Among the diversity of crowd motion models available in the literature, those of interest when dealing with MFGs are the \emph{macroscopic} models, in which the crowd at a given time $t$ is represented by a probability measure $m_t$ on the space of possible positions $X$, which evolves according to some conservation law, typically a continuity equation of the form $\partial_t m + \diverg(m V) = 0$, where $V$ is the velocity field followed by the agents. While most macroscopic crowd motion models consider a given velocity field $V$ constructed from modeling assumptions, the mean field game approach consists instead in considering that each agent will choose their trajectory by solving some optimal control problem, the velocity field $V$ being a consequence of the optimal choices of the agents. In other words, MFG models for crowd motion usually try to capture strategic choices of the crowd based on the rational anticipation by an agent of the behavior of others.

Up to the author's knowledge, the first work to be fully dedicated to a mean field game model for crowd motion is \cite{Lachapelle2011Mean}, which proposes an MFG model for a two-population crowd with trajectories perturbed by additive Brownian motion and considers both their stationary distributions and their evolution on a prescribed time interval. Since then, many other works have studied MFG models for (or related to) crowd motion with a diversity of modeling perspectives and assumptions, such as \cite{Burger2013Mean, Cardaliaguet2016First, Benamou2017Variational, Bagagiolo2022Optimal, Cristiani2023Generalized, Ducasse2022Second, Mazanti2019Minimal, Dweik2020Sharp, Sadeghi2022Multi, Sadeghi2022Nonsmooth, Sadeghi2021Characterization}. A key feature for an MFG model for crowd motion is that it should take into account the difficulty of pedestrians of passing through congested regions, and several works in the literature do that by a suitable penalization term in the cost function, yielding the so-called \emph{MFGs of congestion} \cite{Evangelista2018First, Achdou2018Mean, Gomes2015Existence}. These games typically consider costs which include a product of the form $m_t^\alpha \abs{\dot\gamma(t)}^\beta$ (for some exponents $\alpha, \beta > 0$), where $m_t$ is evaluated at or in a neighborhood of $\gamma(t)$, meaning that high velocities are costly, and that they are even more costly in the presence of high concentrations. This is not the point of view considered in this paper, which models congestion as a constraint on the velocity of agents through \eqref{eq:intro-velocity-bound}.

In order to properly model congestion, the function $K$ in \eqref{eq:intro-velocity-bound} should compute $K(\mu, x)$ for some distribution of agents $\mu$ and some position $x$ by evaluating $\mu$ at or around $x$ and giving as a result some nonincreasing function of this evaluation, meaning that the maximal speed of an agent is a nonincreasing function of some ``average density'' around $x$. A natural choice, for instance, would be $K(\mu, x) = \max(0, 1 - \mu(x))$ if $\mu$ is absolutely continuous with respect to the Lebesgue measure, with a continuous density $\mu$, and $+\infty$ otherwise, which is the one adopted in Hughes' model for crowd motion \cite{Hughes2002Continuum}. However, due to the lack of regularity of this choice of $K$ --- and of any other choice of $K$ depending on a local evaluation of $\mu$ at a given position $x$ ---, there are no general mathematical results on the existence of solutions for neither Hughes' nor MFG models, apart from some results for Hughes' model in dimension $1$ \cite{Amadori2014Existence, DiFrancesco2011Hughes, Andreianov2023Existence}. The interested reader can find an overview of mathematical results for the Hughes' model in the recent survey \cite{Amadori2023Mathematical} and further links between Hughes' model and MFGs in \cite{GhattassiNonseparable}. We also highlight that the choice of $K$ in Hughes' model is an indirect way to model a density constraint, since it imposes that agents can only move with a positive speed in the domain $\{x \in X \suchthat \mu(x) < 1\}$. We refer to \cite{Meszaros2015Variational, Cardaliaguet2016First, Cesaroni2021One, Daudin2023Optimal, Lavenant2019New} for mean field games and multi-agent optimal control problems with density constraints.

To avoid issues concerning the regularity of $K$, our model is motivated by the case where $K$ is nonlocal. A typical kind of $K$ one can keep in mind is
\begin{equation}
\label{IntroK}
K(\mu, x) = \kappa\left(\int_{X} \chi(x, y) \eta(y) \diff \mu(y)\right),
\end{equation}
where $\chi\colon X \times X \to [0, +\infty)$ is a kernel, $\eta\colon X \to [0, +\infty)$ may serve as a weight on $X$ or as a cut-off function to discount some part of $X$, and the nonincreasing function $\kappa\colon [0, +\infty) \to [0, +\infty)$ provides the maximal speed in terms of the average density computed by the integral. Even though the results presented in this paper do not assume a particular form for $K$, most of our assumptions are introduced having in mind that they should be verified for \eqref{IntroK} under suitable regularity assumptions on $\kappa$, $\chi$, and $\eta$.

The MFG model considered here is an extension of those treated in \cite{Mazanti2019Minimal, Dweik2020Sharp, Sadeghi2022Multi, Sadeghi2022Nonsmooth, Sadeghi2021Characterization} and is also related to \cite{Ducasse2022Second}, which considers a second-order version of the model. We refer to those references, and in particular to the introduction of \cite{Ducasse2022Second}, for further discussion on the model and more details on its relations to other mean field games considered in the literature. A detailed comparison between our model and our results with those from \cite{Mazanti2019Minimal, Dweik2020Sharp, Sadeghi2022Multi, Sadeghi2022Nonsmooth, Sadeghi2021Characterization} is provided in Section~\ref{sec:compare}.

In this work, as in \cite{Mazanti2019Minimal, Dweik2020Sharp, Sadeghi2022Multi, Sadeghi2022Nonsmooth, Sadeghi2021Characterization}, we adopt a Lagrangian framework for describing equilibria of the mean field game we consider: instead of representing the movement of the agents as a time-dependent probability measure $m_t$ on $X$, we make use instead of a probability measure $Q$ on the set of all possible trajectories of the agents in $X$. The Lagrangian formulation is a classical approach in optimal transport problems (see, e.g., \cite{Santambrogio2015Optimal, Ambrosio2005Gradient}), which has been used for instance in \cite{Brenier1989Least} to study incompressible flows, in \cite{Carlier2008Optimal} for Wardrop equilibria in traffic flow, or in \cite{Bernot2009Optimal} for branched transport problems. The use of the Lagrangian approach in mean field games dates back at least to \cite{Cardaliaguet2016First, Cardaliaguet2015Weak}, and since then it has been used in several works, such as \cite{Benamou2017Variational, Cannarsa2018Existence, Cannarsa2019C11, Cannarsa2021Mean, Fischer2021Asymptotic, Mazanti2019Minimal, Dweik2020Sharp, Sadeghi2022Multi, Sadeghi2022Nonsmooth, Sadeghi2021Characterization}. Some of these references call equilibria in a Lagrangian framework as \emph{Lagrangian equilibria} but, for simplicity, we will only use the term \emph{equilibria} in the sequel of this paper.

In addition to proving existence of equilibria, we will also consider in this paper the asymptotic behavior at an equilibrium of the probability measure $m_t$ describing the distribution of agents at time $t$ as $t \to +\infty$. The study of the asymptotic behavior of mean field games and optimal control problems is a classical problem that was addressed in several works in the literature. In the context of optimal control, a standard question, addressed for instance in \cite{BardiAsymptotic, Lasry1975Controle, Grune1998Relation, Davini2016Convergence, Ziliotto2019Convergence, Roquejoffre2001Convergence, Cannarsa2022Asymptotic, Barles2000Large}, is to understand the large-time behavior of the value function of an optimal control problem, and studying more precisely whether the ergodic limit of such a value function converges to the value function of a suitable stationary problem, described as the solution of a stationary Hamilton--Jacobi equation. As for mean field games, a frequent question, addressed for instance in \cite{Cardaliaguet2013Long1, Cardaliaguet2013Long2, Cardaliaguet2021Ergodic, Masoero2019Long, Cardaliaguet2019Long, Bardi2024Long}, is that of the average behavior of equilibria of games on a finite time interval $[0, T]$ as the time horizon $T$ tends to $+\infty$, most works being interested in whether such a limit can be characterized by a system of stationary PDEs, which can itself be interpreted as a stationary mean field game in infinite horizon. Some works also study the so-called \emph{turnpike property} \cite{Porretta2018Turnpike, Geshkovski2022Turnpike, Cirant2021Long, Trelat2018Integral, Sun2024Turnpike}, which consists on the fact that, for optimal control problems or mean field games in a large time interval $[0, T]$, the behavior of the solution in an interval of the form $[\varepsilon, T - \varepsilon]$ for some suitable $\varepsilon > 0$ can be approached by the solution of a stationary problem. The asymptotic analysis carried out in this paper is more closely related to the classical asymptotic analysis of dynamical systems \cite{Hartman1964Ordinary, Robinson1999Dynamical}, as it can be seen as the computation of the $\omega$-limit set of optimal trajectories, the relation with the asymptotic analysis of optimal control problems and mean field games coming from the fact that the dynamics of the distribution of agents $t \mapsto m_t$ is obtained through the solution of optimal control problems.

The paper is organized as follows. Section~\ref{sec:notation} sets the notation and provides definitions important for the sequel of the paper, such as those of Wasserstein distance and metric derivative. The mean field game model considered in this paper is described in details in Section~\ref{sec-first-model}, which also considers an associated optimal control problem. We also provide, in Section~\ref{sec-first-model}, the main assumptions used in the paper, and compare our setting and our results to those of \cite{Mazanti2019Minimal, Dweik2020Sharp, Sadeghi2022Multi, Sadeghi2022Nonsmooth, Sadeghi2021Characterization}. Preliminary properties of the mean field game and the optimal control problem considered in the paper are provided in Section~\ref{sec-prelim}. Finally, our main results are presented in Section~\ref{sec:main}: existence of equilibria of the game is shown in Section~\ref{sec:existence}, the asymptotic behavior of the distribution of agents $m_t$ as $t \to +\infty$ is the subject of Section~\ref{sec:asymptotic}, and the dependence of equilibria and asymptotic limits on the initial distribution of agents is studied in Section~\ref{sec:dependence-m0}.

\section{Notation and definitions}
\label{sec:notation}

In this paper, we denote by $\mathbb N$, $\mathbb N^\ast$, $\mathbb R_+$, and $\mathbb R_+^\ast$ the sets of nonnegative integers, positive integers, nonnegative real numbers, and positive real numbers, respectively. Given a subset $A$ of a topological space $X$, its boundary and closure are denoted by $\partial A$ and $\overline A$, respectively. The open and closed balls in a metric space $X$ centered at $x \in X$ and with radius $r \geq 0$ are denoted by $B_{X}(x, r)$, $\overline B_{X}(x, r)$, respectively. Denoting by $\dist$ the metric of $X$, given $x \in X$ and $A \subset X$, we write $\dist(x, A)$ for the distance from the point $x$ to the set $A$, defined by $\dist(x, A) = \inf_{y \in A} \dist(x, y)$. We use $\abs{x}$ to denote the Euclidean norm of a vector $x \in \mathbb R^d$.

Given two metric spaces $X$ and $Y$, we use $\mathcal C(X, Y)$ (respectively, $\Lip(X, Y)$) to denote the set of continuous (respectively, Lipschitz continuous) functions from $X$ to $Y$. Given $c \geq 0$, we define $\Lip_c(X, Y)$ as the subset of $\Lip(X, Y)$ made of those functions whose Lipschitz constant is upper bounded by $c$. For ease of notation, we write simply $\mathcal C_X$ (respectively, $\Lip(X)$, $\Lip_c(X)$) for $\mathcal C(\mathbb R_+, X)$ (respectively, $\Lip(\mathbb R_+, X)$, $\Lip_c(\mathbb R_+, X)$). We always assume that $\mathcal C_X$, $\Lip(X)$, and $\Lip_c(X)$ are endowed with the topology of uniform convergence on compact sets and we recall that, when $(X, \dist)$ is a complete and separable metric space, $\mathcal C_{X}$ is a Polish space (see, e.g., \cite[Corollary~3, page~X.9; Corollary, page~X.20; and Corollary, page~X.25]{Bourbaki2007Topologie}). Whenever needed, we endow $\mathcal C_{X}$ with the complete distance
\begin{equation*}
\dist_{\mathcal C_{X}}(\gamma_1, \gamma_2) = \sum_{n=1}^\infty \frac{1}{2^n} \sup_{t \in [0, n]}\dist(\gamma_1(t), \gamma_2(t)).
\end{equation*}
Recall that, if $X$ is compact, then, thanks to Arzel\`a--Ascoli theorem \cite[Corollary~3, page~X.19]{Bourbaki2007Topologie}, $\Lip_c(X)$ is compact. For $t \in \mathbb R_+$, we denote by $e_t\colon \mathcal C_{X} \to X$ the evaluation map given by $e_t(\gamma) = \gamma(t)$.

We will make use in this paper of set-valued maps, and we denote $f\colon A \rightrightarrows B$ to say that $f$ is a set-valued map defined on $A$ and taking as values subsets of $B$, i.e., $f(a) \subset B$ for every $a \in A$.

Given a Polish space $X$, the set of all Borel probability measures defined on $X$ is denoted by $\mathcal P(X)$. The support of a measure $\mu \in \mathcal P(X)$ is denoted by $\supp(\mu)$. Given two measurable spaces $X$ and $Y$, a measurable map $f\colon X \to Y$, and a measure $\mu$ in $X$, we use $f_{\#} \mu$ to denote the pushforward of $\mu$ through $f$, i.e., the measure on $Y$ defined by $f_{\#} \mu (A) = \mu(f^{-1}(A))$ for every measurable set $A \subset Y$.

We will always consider in the sequel that, for a Polish space $X$, the set $\mathcal P(X)$ is endowed with the topology of weak convergence of measures. Recall that, by the Portmanteau theorem (see, e.g., \cite[Chapter~1, Theorem~2.1]{Billingsley1999Convergence}), a sequence $(\mu_n)_{n \in \mathbb N}$ in $\mathcal P(X)$ converges weakly to some $\mu \in \mathcal P(X)$ if and only if one of the following equivalent conditions is satisfied:
\begin{itemize}
\item $\displaystyle\lim_{n \to +\infty} \int_X f(x) \diff \mu_n(x) = \int_X f(x) \diff \mu(x)$ for every continuous and bounded function $f\colon X \to \mathbb R$;
\item $\displaystyle\limsup_{n \to +\infty} \mu_n(F) \leq \mu(F)$ for every closed set $F \subset X$;
\item $\displaystyle\liminf_{n \to +\infty} \mu_n(G) \geq \mu(G)$ for every open set $G \subset X$.
\end{itemize}

For a Polish space $X$ endowed with a complete metric $\dist$, we define, for $p \in [1, +\infty)$,
\[
\mathcal{P}_p(X) = \left\{ \mu \in \mathcal{P}(X) \suchthat \int_{X} \dist(x, \overline x)^p \diff \mu(x) < +\infty \text{ for some }\overline x \in X\right\}.
\]
Clearly, if $\int_{X} \dist(x, \overline x)^p \diff \mu(x) < +\infty$ for some $\overline x \in X$, then the same is true also for every $\overline x \in X$. We endow $\mathcal{P}_p(X)$ with the usual Wasserstein distance $\mathbf{W}_p$, defined by
\begin{equation}
\label{eq-defi-Wasserstein}
\mathbf{W}_p(\mu, \nu) = \inf \left\{ \int_{X \times X} \dist(x, y)^p \diff\lambda(x,y) \suchthat \lambda \in \Pi(\mu, \nu) \right\}^{1/p},
\end{equation}
where $\Pi(\mu,\nu)= \left\{\lambda \in \mathcal{P}(X \times X) \suchthat {\pi_{1}}_{\#} \lambda=\mu,\, {\pi_{2}}_{\#} \lambda=\nu \right\}$ and $\pi_1,\, \pi_2\colon X \times X \to X$ denote the canonical projections onto the first and second factors of the product $X \times X$, respectively.

Let $X$ be a metric space with metric $\dist$, $a, b \in \mathbb R$ with $a < b$, and $\gamma\colon (a, b) \to X$. The \emph{metric derivative}\label{AbsDotGamma} of $\gamma$ at a point $t \in (a, b)$ is defined by
\[\abs{\dot\gamma}(t) = \lim_{s \to t} \frac{\dist(\gamma(s), \gamma(t))}{\abs{s - t}}\]
whenever this limit exists. Recall that, if $\gamma$ is absolutely continuous, then $\abs{\dot\gamma}(t)$ exists for almost every $t \in (a, b)$ (see, e.g., \cite[Theorem~1.1.2]{Ambrosio2005Gradient}).

\section{The model}
\label{sec-first-model}

\subsection{Description of the model}

Let $(X, \dist)$ be a complete and separable metric space, $\Gamma \subset X$ be nonempty and closed, $K\colon \mathcal P(X) \times X \to \mathbb R_+$, $g\colon \Gamma \to \mathbb R_+$, and $m_0 \in \mathcal P(X)$. We consider in this paper the following mean field game, denoted by $\MFG(X, \Gamma, K, g, m_0)$. Agents evolve in $X$, their distribution at time $t \in \mathbb R_+$ being given by a probability measure $m_t \in \mathcal P(X)$, with $m_0 \in \mathcal P(X)$ being given. The goal of each agent is to minimize the sum of the time at which they reach the exit $\Gamma$ with the exit cost $g$ computed at the position at which they first reach $\Gamma$, and we assume that the speed of an agent at a position $x$ in time $t$ is bounded by $K(m_t, x)$.

Before providing a more mathematically precise definition of this mean field game and its equilibria, we first introduce an associated optimal control problem where agents evolving in $X$ want to reach $\Gamma$ while minimizing the sum of their arrival time and a function of their arrival position, their speed being bounded by some time- and state-dependent function $k\colon \mathbb R_+ \times X \to \mathbb R_+$. Here $k$ will not depend on the density of the agents, and we consider instead that the dependence of $k$ with respect to time is known. This optimal control problem is denoted in the sequel by $\OCP(X, \Gamma, k, g)$.

\begin{definition}[$\OCP(X, \Gamma, k, g)$]
\label{DefiOCP}
Let $(X, \dist)$ be a complete and separable metric space, $\Gamma \subset X$ be nonempty and closed, and $k\colon \mathbb R_+ \times X \to \mathbb R_+$ and $g\colon \Gamma \to \mathbb R_+$ be continuous.
\begin{enumerate}
\item\label{DefiOCPAdm} A curve $\gamma \in \Lip(X)$ is said to be $k$-\emph{admissible} for $\OCP(X, \Gamma, k, g)$ if its metric derivative $\abs{\dot\gamma}$ satisfies $\abs{\dot\gamma}(t) \leq k(t, \gamma(t))$ for almost every $t \in \mathbb R_+$. The set of all $k$-admissible curves is denoted by $\Adm(k)$.

\item\label{DefiOCPTau} Let $t_0 \in \mathbb R_+$. The \emph{first exit time} after $t_0$ of a curve $\gamma \in \mathcal C_X$ is the number $\tau(t_0, \gamma) \in [t_0, +\infty]$ defined by
\begin{equation}
\label{eq-defi-tau}
\tau(t_0, \gamma) = \inf\{t \geq 0 \suchthat \gamma(t + t_0) \in \Gamma\},
\end{equation}
with the convention $\inf\emptyset = +\infty$.

\item The \emph{final cost function} is the function $G \colon \mathbb R_+ \times \mathcal C_X \to [0, +\infty]$ defined, for $(t_0, \gamma) \in \mathbb R_+ \times \mathcal C_X$, by
\[
G(t_0, \gamma) = \begin{dcases*}
g(\gamma(t_0 + \tau(t_0, \gamma))) & if $\tau(t_0, \gamma) < +\infty$, \\
+\infty & otherwise.
\end{dcases*}
\]

\item\label{DefiOCPOptimalTraj} Let $t_0 \in \mathbb R_+$ and $x_0 \in X$. A curve $\gamma \in \Lip(X)$ is said to be an \emph{optimal curve} or \emph{optimal trajectory} for $(k, g, t_0, x_0)$ if $\gamma \in \Adm(k)$, $\gamma(t) = x_0$ for every $t \in [0, t_0]$, $\tau(t_0, \gamma) < +\infty$, $\gamma(t) = \gamma(t_0 + \tau(t_0, \gamma)) \in \Gamma$ for every $t \in [t_0 + \tau(t_0, \gamma), \allowbreak +\infty)$, and
\begin{equation}
\label{EqMinimalTime}
\tau(t_0, \gamma) + G(t_0, \gamma) = \min_{\substack{\beta \in \Adm(k) \\ \beta(t_0) = x_0}} \tau(t_0, \beta) + G(t_0, \beta).
\end{equation}
The set of all optimal curves for $(k, g, t_0, x_0)$ is denoted by $\Opt(k, g, t_0, x_0)$.

\item\label{def-value-function} The \emph{value function} of the optimal control problem $\OCP(X, \Gamma, k, g)$ is the function $\varphi\colon \mathbb R_+ \times X \to \mathbb R_+ \cup \{+\infty\}$ defined for $(t_0, x_0) \in \mathbb R_+ \times X$ by
\begin{equation}
\label{value function}
\varphi(t_0, x_0) = \inf_{\substack{\gamma \in \Adm(k) \\ \gamma(t_0) = x_0}} \tau(t_0, \gamma) + G(t_0, \gamma).
\end{equation}
\end{enumerate}
\end{definition}

\begin{remark}
\label{remk-tau}
When $\tau(t_0, \gamma) < +\infty$, the infimum in \eqref{eq-defi-tau} is a minimum, since $\gamma$ is continuous and $\Gamma$ is closed. In addition, it follows immediately from \eqref{eq-defi-tau} that, for every $h \geq 0$, we have
\[
t_0 + \tau(t_0, \gamma) \leq t_0 + h + \tau(t_0 + h, \gamma),
\]
with equality if $\gamma(t) \notin \Gamma$ for every $t \in [t_0, t_0 + h)$.
\end{remark}

\begin{remark}
\label{RemkControlSyst}
If $X$ is the closure of an open subset of $\mathbb R^d$, the constraint $\abs{\dot\gamma}(t) \leq k(t, \gamma(t))$, $t \geq 0$, imposed on a $k$-ad\-mis\-sible curve can be interpreted as the fact that $\gamma$ is a solution to a control system, justifying thus the optimal control terminology used in Definition~\ref{DefiOCP}. Indeed, a curve $\gamma$ is $k$-admissible if and only if there exists a measurable function $u\colon \mathbb R_+ \to \overline B_{\mathbb R^d}(0, 1)$ such that
\begin{equation}
\label{AdmissibleIsControlSystem}
\dot\gamma(t) = k(t, \gamma(t)) u(t).
\end{equation}
\end{remark}

\begin{remark}
The fact that we only consider trajectories $\gamma \in \mathcal C_X$ implies that $\gamma(t) \in X$ for every $t \in \mathbb R_+$, and this condition can often be interpreted as a \emph{state constraint} of the optimal control problem $\OCP(X, \Gamma, k, g)$. Indeed, in many applications, $X$ is a subset of a larger space $Y$ (for instance, $X = \overline\Omega$ and $Y = \mathbb R^d$ for some nonempty bounded open set $\Omega \subset \mathbb R^d$), and the condition $\gamma(t) \in X$ can be interpreted as the constraint of preventing $\gamma$ to leave $X$ and go into $Y \setminus X$.
\end{remark}

Given a mean field game $\MFG\allowbreak(X,\allowbreak \Gamma,\allowbreak K, g, m_0)$ and denoting by $m_t \in \mathcal P(X)$ the distribution of agents at time $t \geq 0$, we assume that each agent of the game solves the optimal control problem $\OCP(X, \Gamma, k, g)$ with $k\colon \mathbb R_+ \times X \to \mathbb R_+$ given by $k(t, x) = K(m_t, x)$. An equilibrium of $\MFG(X, \Gamma, K, g, m_0)$ is a situation in which the solutions of these optimal control problems by all agents will induce an evolution of the distribution of agents that coincides with the one given by $m_t$, $t \geq 0$. More precisely, we will use in this paper the following definition.

\begin{definition}[Equilibrium of $\MFG(X, \Gamma, K, g, m_0)$]
\label{DefiEquilibriumMFG}
Let $(X, \dist)$ be a complete and separable metric spa\-ce, $\Gamma \subset X$ be nonempty and closed, and $K\colon \mathcal P(X) \times X \to \mathbb R_+$ and $g \colon \Gamma \to \mathbb R_+$ be continuous functions. Let $m_0 \in \mathcal P(X)$, $Q \in \mathcal P(\mathcal C_X)$, and define $k\colon \mathbb R_+ \times X \to \mathbb R_+$ for $(t, x) \in \mathbb R_+ \times X$ by $k(t, x) = K({e_t}_\# Q, x)$.
\begin{enumerate}
\item\label{DefiWeakEquilibriumMFG} The measure $Q \in \mathcal P(\mathcal C_X)$ is said to be a \emph{weak Lagrangian equilibrium} (or simply \emph{weak equilibrium}) of $\MFG(X, \Gamma,\allowbreak K, g, m_0)$ if ${e_0}_\# Q = m_0$ and $Q$-almost every $\gamma \in \mathcal C_X$ is an optimal curve for $(k, g, 0, \gamma(0))$.
\item\label{DefiStrongEquilibriumMFG} The measure $Q \in \mathcal P(\mathcal C_X)$ is said to be a \emph{strong Lagrangian equilibrium} (or simply \emph{strong equilibrium}) of $\MFG(X, \Gamma,\allowbreak K, g, m_0)$ if ${e_0}_\# Q = m_0$ and every $\gamma \in \supp(Q)$ is an optimal curve for $(k, g, 0, \gamma(0))$.
\end{enumerate}
\end{definition}

To simplify the notation, given $K\colon \mathcal P(X) \times X \to \mathbb R_+$ and $m\colon \mathbb R_+ \to \mathcal P(X)$, we define $k\colon \mathbb R_+ \times X \to \mathbb R_+$ by $k(t, x) = K(m_t, x)$ for $(t, x) \in \mathbb R_+ \times X$ and say that $\gamma \in \Lip(X)$ is $m$-admissible for $\MFG(X, \Gamma, K, g, m_0)$ if it is $k$-admissible for $\OCP(X, \Gamma, k, g)$, denoting $\Adm(k)$ simply by $\Adm(m)$. Given $(t_0, x_0) \in \mathbb R_+ \times X$, we say that $\gamma \in \Lip(X)$ is an optimal trajectory for $(m, g, t_0, x_0)$ if it is an optimal trajectory for $(k, g, t_0, x_0)$ for the optimal control problem $\OCP(X, \Gamma, k, g)$, and denote the set of optimal trajectories for $(m, g, t_0, x_0)$ by $\Opt(m, g, t_0, x_0)$. Given $Q \in \mathcal P(\mathcal C_X)$, we consider the time-dependent measure $m^Q\colon \mathbb R_+ \to \mathcal P(X)$ given by $m^Q_t = {e_t}_\# Q$, and denote $\Adm(m^Q)$ and $\Opt(m^Q, g, t_0, x_0)$ simply by $\Adm(Q)$ and $\Opt(Q, g, t_0, x_0)$, respectively.

\begin{remark}
\label{remk-equilibria-strong-weak}
Since $Q(\supp(Q)) = 1$, any strong equilibrium is also a weak equilibrium. The converse turns out to be true for the model we consider in this paper, as we will prove later in Corollary~\ref{coro-strong-iff-weak}, a result that generalizes \cite[Remark~4.6]{Dweik2020Sharp} and \cite[Proposition~3.7]{Sadeghi2022Nonsmooth} to the present model.
\end{remark}

\begin{remark}
\label{RemkOptimalTrajectoriesRemainStoppedAfterFinalTime}
The cost $\tau_0(t_0, \gamma) + G(t_0, \gamma)$ only takes into account the values of $\gamma$ on the interval $[t_0, t_0 + \tau_0(t_0, \gamma)]$ and, in particular, if $\gamma \in \Adm(k)$ is a minimizer of $\tau_0(t_0, \cdot) + G(t_0, \cdot)$ with the constraint $\gamma(t_0) = x_0$, then any other trajectory $\widetilde\gamma \in \Adm(k)$ coinciding with $\gamma$ in $[t_0, t_0 + \tau_0(t_0, \gamma)]$ is also a minimizer of the same cost with this constraint. In order to avoid ambiguity on the behavior of minimizers before $t_0$ or after $t_0 + \tau_0(t_0, \gamma)$, Definition~\ref{DefiOCP}\ref{DefiOCPOptimalTraj} defines an \emph{optimal trajectory} to this minimization problem as being necessarily constant in the intervals $[0, t_0]$ and $[t_0 + \tau_0(t_0, \gamma), +\infty)$.

From the point of view of the mean field game $\MFG(X, \Gamma, K, g, m_0)$, the above choice leads to a concentration of agents on the target set $\Gamma$: interpreting the game as a crowd motion model, this would mean that agents that leave the domain $X$ through $\Gamma$ remain stopped at their arrival position at $\Gamma$, which creates congestion in $\Gamma$. In order to consider that agents of the game ``disappear'' when they reach $\Gamma$, one may consider, for modeling purposes, that, for $K$ given by \eqref{IntroK}, the function $\eta$ is a cut-off function, equal to $1$ everywhere on $X$ except on a neighborhood of $\Gamma$ and vanishing at $\Gamma$. Notice, however, that such an assumption on $K$ is not necessary for the results proved in this paper.
\end{remark}

\subsection{Main assumptions and their consequences}
\label{sec-existence}

Let us now present the main assumptions needed in the sequel. The following assumptions are common to $\MFG(X, \Gamma, K, g, m_0)$ and $\OCP(X, \Gamma, k, g)$.
\begin{hypothesis}
\item\label{Hypo-X-SigmaCompact} $(X, \dist)$ is a metric space and there exists $\mathbf 0 \in X$ such that, for every $R > 0$, the closed ball $\overline B_X(\mathbf 0, R)$ is compact.
\item\label{Hypo-Gamma} $\Gamma \subset X$ is nonempty and closed.
\item\label{Hypo-g} The function $g\colon \Gamma \to \mathbb R_+$ is Lipschitz continuous, with Lipschitz constant $L_g$.
\item\label{Hypo-X-dist} There exists $D > 0$ such that, for every $x, y \in X$, there exist $T \in [0, D\dist(x, y)]$ and $\gamma \in \Lip([0, T],\allowbreak X)$ such that $\gamma(0) = x$, $\gamma(T) = y$, and $\abs{\dot\gamma}(t) = 1$ for almost every $t \in [0, T]$.
\end{hypothesis}

The following assumptions are specific to $\OCP(X, \Gamma, k, g)$.
\begin{hypothesis}[resume]
\item\label{HypoOCP-k-Bound} The function $k\colon \mathbb R_+ \times X \to \mathbb R_+$ is continuous and there exist $K_{\min}, K_{\max} \in \mathbb R_+^\ast$ such that, for all $(t, x) \in \mathbb R_+ \times X$, one has $k(t, x) \in [K_{\min}, K_{\max}]$.
\item\label{HypoOCP-k-Lip} For every $R > 0$, there exists $L_R > 0$ such that, for every $t \in \mathbb R_+$ and $x_1, x_2 \in \overline B_X(\mathbf 0, R)$, we have
\[
\abs{k(t, x_1) - k(t, x_2)} \leq L_R \dist(x_1, x_2).
\]
\item\label{HypoOCP-g-compatible} Hypotheses~\ref{Hypo-g} and \ref{HypoOCP-k-Bound} are satisfied with $L_g K_{\max} < 1$.
\end{hypothesis}

Finally, we state the following assumptions, specific to $\MFG(X, \Gamma, K, g, m_0)$.
\begin{hypothesis}[resume]
\item\label{HypoMFG-K-Bound} The function $K\colon \mathcal P(X) \times X \to \mathbb R_+$ is continuous and there exist $K_{\min}, K_{\max} \in \mathbb R_+^\ast$ such that, for all $(\mu, x) \in \mathcal P(X) \times X$, one has $K(\mu, x) \in [K_{\min}, K_{\max}]$.
\item\label{HypoMFG-K-Lip} For every $R > 0$, there exists $L_R > 0$ such that, for every $\mu \in \mathcal P(X)$ and $x_1, x_2 \in \overline B_X(\mathbf 0, R)$, we have
\[
\abs{K(\mu, x_1) - K(\mu, x_2)} \leq L_R \dist(x_1, x_2).
\]
\item\label{HypoMFG-g-compatible} Hypotheses~\ref{Hypo-g} and \ref{HypoMFG-K-Bound} are satisfied with $L_g K_{\max} < 1$.
\end{hypothesis}

Hypotheses~\ref{Hypo-X-SigmaCompact}--\ref{Hypo-X-dist} are the standard assumptions used in most of the results of this paper. Hypothesis~\ref{Hypo-X-SigmaCompact} is inspired by the case where $X$ is a closed subset of a finite-dimensional vector space, and it implies, in particular, that $X$ is $\sigma$-compact, complete, and separable. Clearly, if \ref{Hypo-X-SigmaCompact} is satisfied for some element $\mathbf 0 \in X$, then it is also satisfied with $\mathbf 0$ replaced by any element $x \in X$. In the sequel, we shall always consider that $\mathbf 0$ is a fixed element of $X$.

Hypothesis~\ref{Hypo-X-dist} provides a relation between the distance $\dist(x, y)$ and the length of curves from $x$ to $y$ in $X$, stating that the former is, up to a constant, an upper bound on the latter. This assumption is satisfied when the geodesic metric induced by $\dist$ is equivalent to $\dist$ itself. In particular, it is satisfied if $X$ is a length space. Moreover, \ref{Hypo-X-dist} implies that $X$ is path-connected.

One of the important consequences of \ref{HypoOCP-k-Bound}, which we will use frequently in the sequel, is that it implies that $\Adm(k) \subset \Lip_{K_{\max}}(X)$.

Hypotheses~\ref{HypoOCP-k-Lip} and \ref{HypoMFG-K-Lip} can be reformulated by saying that, on every compact subset of $X$, the functions $k$ and $K$ are Lipschitz continuous with respect to their second variable, uniformly with respect to the first one.

Notice that, under suitable assumptions on $\kappa$, $\chi$, and $\eta$, the function $K$ defined in \eqref{IntroK} satisfies \ref{HypoMFG-K-Bound} and \ref{HypoMFG-K-Lip}. More precisely, we have the following result.

\begin{proposition}
\label{PropFConvolution}
Let $(X, \dist)$ be a complete and separable metric space, $\kappa \in \Lip(\mathbb R_+, \mathbb R_+^\ast)$, $\chi \in \mathcal C(X \times X, \mathbb R_+)$ be bounded, $\eta \in \mathcal C(X, \mathbb R_+)$ be bounded, and define $K\colon \mathcal P(X) \times X \to \mathbb R_+$ by \eqref{IntroK}. Assume that, for every $R > 0$, there exists $L_{\chi, R} > 0$ such that, for every $x_1, x_2 \in \overline B_X(\mathbf 0, R)$ and $y \in X$, we have
\[
\abs{\chi(x_1, y) - \chi(x_2, y)} \leq L_{\chi, R} \dist(x_1, x_2).
\]
Then $K$ satisfies \ref{HypoMFG-K-Bound} and \ref{HypoMFG-K-Lip}.
\end{proposition}

A result very similar to Proposition~\ref{PropFConvolution} was shown in \cite[Proposition~3.1]{Mazanti2019Minimal} in the case where $X = \overline\Omega$ for some nonempty bounded open set $\Omega \subset \mathbb R^d$, and with additional assumptions on $\chi$ that also allow one to prove that $K$ is Lipschitz continuous when considering the Wasserstein distance $\mathbf W_1$ in $\mathcal P(\overline\Omega)$. We provide here an adaptation of that proof to our current setting.

\begin{proof}[Proof of Proposition~\ref{PropFConvolution}]
Let $E\colon \mathcal P(X) \times X \to \mathbb R_+$ be defined by
\[
E(\mu, x) = \int_{X} \chi(x, y) \eta(y) \diff \mu(y).
\]

Let $R > 0$. For every $\mu \in \mathcal P(X)$ and $x_1, x_2 \in \overline B_X(\mathbf 0, R)$, we have
\[
\abs*{E(\mu, x_1) - E(\mu, x_2)} \leq L_\chi \dist(x_1, x_2) \int_{X} \eta(y) \diff\mu(t).
\]
Since $\eta$ is bounded, the integral in the above expression is bounded independently of $\mu$, and thus $E$ is Lipschitz continuous with respect to its second variable in $\overline B_X(\mathbf 0, R)$, uniformly with respect to the first one. Since $K = \kappa \circ E$, we deduce that $K$ satisfies \ref{HypoMFG-K-Lip}.

Since $\chi$ and $\eta$ are bounded, $E$ is also bounded by some constant $M > 0$. Since $K = \kappa \circ E$ and $\kappa$ is continuous and takes values in $\mathbb R_+^\ast$, we deduce that $K$ takes values in $[K_{\min}, K_{\max}]$, where $K_{\min} = \min_{s \in [0, M]} \kappa(s) > 0$ and $K_{\max} = \max_{s \in [0, M]} \kappa(s) > 0$.

If $(\mu_n)_{n \in \mathbb N}$ is a sequence in $\mathcal P(X)$ converging to some $\mu \in \mathcal P(X)$ and $x \in X$ is given, then, since $y \mapsto \chi(x, y) \eta(y)$ is continuous and bounded, we deduce that $E(\mu_n, x) \to E(\mu, x)$ as $n \to +\infty$. Together with the fact that $(\mu, x) \mapsto E(\mu, x)$ is Lipschitz continuous in $x \in \overline B_X(\mathbf 0, R)$ uniformly in $\mu$ for every $R > 0$, we conclude that $E$ is continuous on $\mathcal P(X) \times X$, and thus, since $\kappa$ is continuous, we deduce that $K$ satisfies \ref{HypoMFG-K-Bound}.
\end{proof}

Hypotheses~\ref{HypoOCP-g-compatible} and \ref{HypoMFG-g-compatible} can be seen as restriction on the Lipschitz constant of $g$ in terms of the upper bound $K_{\max}$ on $k$ or $K$. This is a standard assumption in optimal control problems with free final time and boundary costs (see, e.g., \cite[(8.6) and Remark~8.1.5]{Cannarsa2004Semiconcave} and \cite{Dweik2018Optimal, Dweik2020Sharp}), its importance being the following property, whose proof is straightforward.

\begin{lemma}
\label{LemmTimePlusGIncreases}
Assume that \ref{Hypo-X-SigmaCompact}--\ref{Hypo-g} are satisfied and let $\gamma \in \mathcal C_X$ be Lipschitz continuous, with Lipschitz constant $L_\gamma$ satisfying $L_g L_\gamma < 1$. If $t_1, t_2 \in \mathbb R_+$ are such that $t_1 < t_2$ and $\gamma(t_1), \gamma(t_2) \in \Gamma$, then
\[
t_1 + g(\gamma(t_1)) < t_2 + g(\gamma(t_2)).
\]
\end{lemma}

As a consequence of Lemma~\ref{LemmTimePlusGIncreases}, we also obtain the following result.

\begin{lemma}
\label{lemm-tau-plus-G-lsc}
Assume that \ref{Hypo-X-SigmaCompact}--\ref{Hypo-g} are satisfied and let $C > 0$ be such that $L_g C < 1$. Then the function $(t, \gamma) \mapsto \tau(t, \gamma) + G(t, \gamma)$ is lower semicontinuous on $\mathbb R_+ \times \Lip_C(X)$.
\end{lemma}

\begin{proof}
Let $(t_n, \gamma_n)_{n \in \mathbb N}$ be a sequence in $\mathbb R_+ \times \Lip_C(X)$ converging, in the topology of $\mathbb R_+ \times \mathcal C_X$, to some $(t, \gamma) \in \mathbb R_+ \times \mathcal C_X$. Since $\Lip_C(X)$ is closed, we have $\gamma \in \Lip_C(X)$. We want to prove that
\[
\tau(t, \gamma) + G(t, \gamma) \leq \liminf_{n \to +\infty} \tau(t_n, \gamma_n) + G(t_n, \gamma_n).
\]

There is nothing to prove if the right-hand side of the above inequality is equal to $+\infty$, so we assume, with no loss of generality, that it is finite. We extract a subsequence $(t_{n_k}, \gamma_{n_k})_{k \in \mathbb N}$ of $(t_n, \gamma_n)_{n \in \mathbb N}$ such that $\tau(t_{n_k}, \gamma_{n_k}) + G(t_{n_k}, \gamma_{n_k}) < +\infty$ for every $k \in \mathbb N$ and
\[
\lim_{k \to +\infty} \tau(t_{n_k}, \gamma_{n_k}) + G(t_{n_k}, \gamma_{n_k}) = \liminf_{n \to +\infty} \tau(t_n, \gamma_n) + G(t_n, \gamma_n).
\]
For simplicity, we denote $(t_{n_k}, \gamma_{n_k})_{k \in \mathbb N}$ simply by $(t_n, \gamma_n)_{n \in \mathbb N}$ in the sequel. Also, without loss of generality, we assume that $(\tau(t_n, \gamma_n))_{n \in \mathbb N}$ converges to some $\tau_\ast \in \mathbb R_+$. Then $G(t_n, \gamma_n) = g(\gamma_n(t_n + \tau(t_n, \gamma_n))) \to g(\gamma(t + \tau_\ast))$ as $n \to +\infty$.

Since $\gamma_n(t_n + \tau(t_n, \gamma_n)) \in \Gamma$ for every $n \in \mathbb N$ and $\Gamma$ is closed, we obtain that $\gamma(t + \tau_\ast) \in \Gamma$, and thus $\tau(t, \gamma) \leq \tau_\ast$. Hence $t + \tau(t, \gamma) \leq t + \tau_\ast$ and, by Lemma~\ref{LemmTimePlusGIncreases}, we deduce that $t + \tau(t, \gamma) + G(t, \gamma) \leq t + \tau_\ast + g(\gamma(t + \tau_\ast))$, which is the desired inequality.
\end{proof}

\subsection{Comparison to other works}
\label{sec:compare}

The model considered here can be seen as a generalization of the mean field game models considered in \cite{Mazanti2019Minimal, Dweik2020Sharp, Sadeghi2021Characterization, Sadeghi2022Multi, Sadeghi2022Nonsmooth}, and we now provide a more detailed comparison between them.

First of all, the goal of \cite{Mazanti2019Minimal, Dweik2020Sharp, Sadeghi2021Characterization, Sadeghi2022Multi, Sadeghi2022Nonsmooth} is broader than that of the present paper: in addition to showing existence of equilibria, these papers also aim at proving that equilibria necessarily satisfy a system of partial differential equations, the MFG system, showing also, when possible, the converse statement (namely, that solutions of the MFG system are equilibria). These additional properties require more assumptions than just \ref{Hypo-X-SigmaCompact}--\ref{HypoMFG-g-compatible}; in particular, in order to properly write the partial differential equations in the MFG system, one usually works on the closure of an open subset of $\mathbb R^d$ instead of a metric space $(X, \dist)$ satisfying \ref{Hypo-X-SigmaCompact}. On the other hand, concerning existence of equilibria, all of the references \cite{Mazanti2019Minimal, Dweik2020Sharp, Sadeghi2021Characterization, Sadeghi2022Multi, Sadeghi2022Nonsmooth} make use of more restrictive assumptions than \ref{Hypo-X-SigmaCompact}--\ref{HypoMFG-g-compatible}, which allows for simpler proofs with respect to the ones presented here.

Among these works, \cite{Mazanti2019Minimal, Sadeghi2021Characterization, Sadeghi2022Multi, Sadeghi2022Nonsmooth} use only the notion of weak equilibrium, while \cite{Dweik2020Sharp} uses the notion of strong equilibrium. The fact that both notions coincide for the models treated in these references under suitable assumptions was already observed in \cite[Remark~4.6]{Dweik2020Sharp} and \cite[Proposition~3.7]{Sadeghi2022Nonsmooth}, a result that we also present later in Corollary~\ref{coro-strong-iff-weak} for the more general model we consider here.

As regards the characterization of the asymptotic behavior of equilibria, among \cite{Mazanti2019Minimal, Dweik2020Sharp, Sadeghi2021Characterization, Sadeghi2022Multi, Sadeghi2022Nonsmooth}, only \cite{Sadeghi2022Multi} considers this question. The main reason for that is that \cite{Mazanti2019Minimal, Dweik2020Sharp, Sadeghi2021Characterization, Sadeghi2022Nonsmooth} always consider that $X$ is compact and, in this case, one can easily prove that, at equilibrium, the distribution $m_t$ converges to a limit distribution $m_\infty$ in finite time, i.e., there exists $t_0 \geq 0$ such that $m_t = m_\infty$ for every $t \geq 0$ (see Theorem~\ref{ThmAsymp}\ref{ThmAsymp-FiniteTime} below). The asymptotic behavior of equilibria is studied in \cite{Sadeghi2022Multi} only in the case $X = \mathbb R^d$, and our main result on this question, Theorem~\ref{ThmAsymp}, is an extension of \cite[Theorem~5.8]{Sadeghi2022Multi} to our more general setting. We also provide further consequences of Theorem~\ref{ThmAsymp} in Corollary~\ref{coro:asymp}, which provides explicit convergence rates for mean field games in $\mathbb R^d$, a novelty with respect to \cite{Sadeghi2022Multi}.

Finally, the analysis of the dependence of equilibria and their asymptotic limits on the initial distribution of agents, which is the topic of Section~\ref{sec:dependence-m0}, is a novelty with respect to \cite{Mazanti2019Minimal, Dweik2020Sharp, Sadeghi2021Characterization, Sadeghi2022Multi, Sadeghi2022Nonsmooth}.

Now that we have compared the present paper with \cite{Mazanti2019Minimal, Dweik2020Sharp, Sadeghi2021Characterization, Sadeghi2022Multi, Sadeghi2022Nonsmooth} in what concerns the kind of results that are shown, we proceed to compare the models treated in those references and their assumptions with the setting of this paper.

In \cite{Mazanti2019Minimal}, the authors consider the mean field game $\MFG(X, \Gamma, K, 0, m_0)$, i.e., there is no cost at the exit position of an agent, and each agent minimizes only their time to reach $\Gamma$. Instead of \ref{Hypo-X-SigmaCompact}, the authors assume that $X$ is compact, which simplifies many technical details in the proofs when compared to the presentation of this paper. Further technical simplifications are also possible in \cite{Mazanti2019Minimal} thanks to the fact that \ref{HypoOCP-k-Lip} and \ref{HypoMFG-K-Lip} are replaced by the stronger assumptions of Lipschitz continuity of $k$ and $K$, respectively, when $\mathcal P(X)$ is endowed with the Wasserstein metric $\mathbf W_1$.

The article \cite{Dweik2020Sharp} considers the mean field game $\MFG(\overline\Omega, \partial\Omega, K, g, m_0)$, where $\Omega \subset \mathbb R^d$ is a nonempty bounded open set. Apart from this fact, the other assumptions required for the existence of equilibria in \cite{Dweik2020Sharp} are essentially the same as \ref{Hypo-X-SigmaCompact}--\ref{HypoMFG-g-compatible}. The fact that $X = \overline\Omega$ and $\Gamma = \partial\Omega$ implies that one does not need to consider the state constraint $\gamma(t) \in X$ in the analysis of the optimal control problem $\OCP(\overline\Omega, \partial\Omega, k, g)$, which allows for some simplifications in the proofs in \cite{Dweik2020Sharp} with respect to the ones presented here.

One of the main novelties of \cite{Sadeghi2022Multi} with respect to \cite{Mazanti2019Minimal, Dweik2020Sharp} is to consider a mean field game in the noncompact space $\mathbb R^d$. More precisely, \cite{Sadeghi2022Multi} addresses the mean field game $\MFG(\mathbb R^d, \Gamma, K, 0, m_0)$, where $\Gamma \subset \mathbb R^d$ is a nonempty closed set, and agents only minimize the time to reach $\Gamma$, with no arrival cost. Instead of \ref{HypoMFG-K-Lip}, \cite{Sadeghi2022Multi} works with a more restrictive assumption, requiring global Lipschitz behavior in the space variable $x$. Hence, with respect to \cite{Sadeghi2022Multi}, we replace here $\mathbb R^d$ by a more general metric space $(X, \dist)$ satisfying \ref{Hypo-X-SigmaCompact} and we work under only local Lipschitz behavior of $K$ with respect to its space variable. Note also that \cite{Sadeghi2022Multi} considers multipopulation mean field games, but the generalization of the techniques of this work to the multipopulation setting are straightforward, and we restrict our attention here to the single-population case in order to avoid cumbersome notation.

Finally, \cite{Sadeghi2021Characterization, Sadeghi2022Nonsmooth} consider both the same model, \cite{Sadeghi2022Nonsmooth} being an extended version of the conference paper \cite{Sadeghi2021Characterization}. They both consider $\MFG(\overline\Omega, \Gamma, K, 0, m_0)$, where $\Omega \subset \mathbb R^d$ is a nonempty bounded open set and $\Gamma \subset \overline\Omega$ is nonempty and closed. This model is actually a particular case of that of \cite{Mazanti2019Minimal}, but \cite{Sadeghi2021Characterization, Sadeghi2022Nonsmooth} work under more general assumptions on $k$ and $K$, which are just the particular case of \ref{HypoOCP-k-Lip} and \ref{HypoMFG-K-Lip} when $X$ is compact. The main difference of the present work with respect to those references is thus the fact that we consider here games in a more general metric space $(X, \dist)$ and with a cost $g$ on the exit position.

\section{Preliminary results}
\label{sec-prelim}

Before turning to the study of existence of equilibria of $\MFG(X, \Gamma, k, g, m_0)$ and their asymptotic behavior, we collect some important preliminary properties of the optimal control problem $\OCP(X, \Gamma, k, g)$ and the mean field game $\MFG(X, \Gamma, k, g, m_0)$ that are useful for the proofs of our main results in Section~\ref{sec:main}.

\subsection{Properties of the optimal control problem}

Note that $\OCP(X, \Gamma, k, g)$ is an optimal control problem with free final time, which is a classic subject in the optimal control literature (see, e.g., \cite{Cannarsa2004Semiconcave, Clarke1990Optimization}), but the assumptions \ref{Hypo-X-SigmaCompact}--\ref{HypoOCP-g-compatible} on $\OCP(X, \Gamma, k, g)$ are more general than the ones usually considered in the literature, allowing for more general spaces $X$ and fewer regularity assumptions on $\Gamma$ and $k$. For this reason, we provide here a detailed presentation of the properties of $\OCP(X, \Gamma, k, g)$ in order to highlight which hypotheses are required for each result.

The first result we provide is the following, on the closedness of $\Adm(k)$.

\begin{proposition}
\label{prop-adm-closed}
Consider the optimal control problem $\OCP(X, \Gamma, k, g)$ and assume that \ref{Hypo-X-SigmaCompact} and \ref{HypoOCP-k-Bound} are satisfied. Then the set $\Adm(k)$ is closed.
\end{proposition}

\begin{proof}
Let $(\gamma_n)_{n \in \mathbb N}$ be a sequence in $\Adm(k)$ converging in $\mathcal C_X$ to some $\gamma \in \mathcal C_X$. Note that, since $\Adm(k) \subset \Lip_{K_{\max}}(X)$ and the latter set is closed in $\mathcal C_X$, we deduce that $\gamma \in \Lip_{K_{\max}}(X)$.

By \cite[Theorem~1.1.2]{Ambrosio2005Gradient} and using the fact that $\gamma_n \in \Adm(k)$, we obtain that, for every $n \in \mathbb N$ and $s, t \in \mathbb R_+$ with $s \leq t$, we have
\begin{equation}
\label{eq-gamma_n-admissible}
\dist(\gamma_n(s), \gamma_n(t)) \leq \int_s^t k(r, \gamma_n(r)) \diff r.
\end{equation}
For every $T_0 > 0$, since $\gamma_n \to \gamma$ in $\mathcal C_X$, there exists $R > 0$ such that $\gamma_n(t) \in \overline B_X(\mathbf 0, R)$ for every $n \in \mathbb N$ and $t \in [0, T_0]$. Since $k$ is continuous on $[0, T_0] \times \overline B_X(\mathbf 0, R)$, it is uniformly continuous on this set. Hence, $k(r, \gamma_n(r)) \to k(r, \gamma(r))$ as $n \to +\infty$, uniformly on $r \in [0, T_0]$. Thus, taking the limit as $n \to +\infty$ in \eqref{eq-gamma_n-admissible}, we deduce that
\[
\dist(\gamma(s), \gamma(t)) \leq \int_s^t k(r, \gamma(r)) \diff r
\]
for every $s, t \in [0, T_0]$ with $s \leq t$. Since $T_0 > 0$ is arbitrary, the previous inequality holds true for every $s, t \in \mathbb R_+$ with $s \leq t$. The conclusion now follows from \cite[Theorem~1.1.2]{Ambrosio2005Gradient}.
\end{proof}

Our next result concerns boundedness of the value function and of optimal trajectories on bounded subsets of initial conditions $x_0$ in $X$, which adapts \cite[Proposition~4.4]{Sadeghi2022Multi} to the case of metric space $X$ and an exit cost $g$.

\begin{proposition}
\label{prop-T-psi}
Consider the optimal control problem $\OCP(X, \Gamma, k, g)$ and assume that \ref{Hypo-X-SigmaCompact}--\ref{HypoOCP-k-Bound} and \ref{HypoOCP-g-compatible} are satisfied. Then there exist two nondecreasing maps with linear growth $\psi,\, T\colon\allowbreak\mathbb{R}_+ \allowbreak \to \mathbb{R}_+$ depending only on $\mathbf 0$, $\Gamma$, $g$, $D$, $K_{\min}$, and $K_{\max}$ such that, for every $R>0$ and $(t_0, x_0) \in \mathbb R_+ \times \overline B_X(\mathbf 0, R)$, we have $\varphi(t_0, x_0) \leq T(R)$ and, for every $\gamma \in \Opt(k, g, t_0, x_0)$, we have $\gamma(t) \in \overline B_X(\mathbf 0, \psi(R))$ for every $t \ge 0$.
\end{proposition} 

\begin{proof}
Fix $y_0 \in \Gamma$ and let $G_0 = g(y_0)$. Let $R > 0$ and $(t_0, x_0) \in \mathbb R_+ \times \overline B_X(\mathbf 0, R)$. Let $T_0 \in [0, D \dist(x_0, y_0)]$ and $\widetilde\gamma\in \Lip([0, T_0], X)$ be such that $\widetilde\gamma(0) = x_0$, $\widetilde\gamma(T_0) = y_0$, and $\abs{\dot{\widetilde\gamma}}(t) = 1$ for almost every $t \in [0, T_0]$. Set $t_1 = t_0 + \frac{T_0}{K_{\min}}$ and let $\gamma\colon \mathbb R_+ \to X$ be defined by $\gamma(t) = x_0$ for $t \in [0, t_0]$, $\gamma(t) = \widetilde\gamma(K_{\min}(t - t_0))$ for $t \in [t_0, t_1]$, and $\gamma(t) = y_0$ for $t \geq t_1$. Then $\gamma \in \Adm(k)$, $\gamma(t_0) = x_0$, and $\tau(t_0, \gamma) \leq \frac{T_0}{K_{\min}}$. By Lemma~\ref{LemmTimePlusGIncreases}, we have
\[
t_0 + \tau(t_0, \gamma) + G(t_0, \gamma) \leq t_1 + g(y_0).
\]
Hence,
\[
\varphi(t_0, x_0) \leq \tau(t_0, \gamma) + G(t_0, \gamma) \leq G_0 + \frac{D \dist(x_0, y_0)}{K_{\min}} \leq G_0 + \frac{D \dist(\mathbf 0, y_0)}{K_{\min}} + \frac{D R}{K_{\min}},
\]
and thus the first part of the conclusion holds true with $T(R) = G_0 + \frac{D \dist(\mathbf 0, y_0)}{K_{\min}} + \frac{D R}{K_{\min}}$.

As for the second part of the statement, take $R>0$, $(t_0, x_0) \in \mathbb R_+ \times \overline B_X(\mathbf 0, R)$, and $\gamma \in \Opt(k, g, t_0, x_0)$. It suffices to prove that $\gamma(t) \in \overline B_X(\mathbf 0, \psi(R))$ for every $t \in [t_0, t_0 + \tau(t_0, \gamma)]$, since $\gamma$ is constant on $[0, t_0]$ and on $[t_0 + \tau(t_0, \gamma), +\infty)$. For any such $t$, using the facts that $\gamma \in \Adm(k) \subset \Lip_{K_{\max}}(X)$ and $\tau(t_0, \gamma) \leq \varphi(t_0, x_0) \leq T(R)$, we have
\begin{equation}
\label{eq-bound-admissible}
\dist(\mathbf 0, \gamma(t)) \leq \dist(\gamma(t), \gamma(t_0)) + \dist(\mathbf 0, x_0) \leq K_{\max} (t - t_0) + R \leq K_{\max} T(R) + R,
\end{equation}
and the conclusion holds true with $\psi(R) = K_{\max} T(R) + R$.
\end{proof}

We can now turn to the question of existence of optimal trajectories for $\OCP(X, \Gamma, k, g)$. The proof of our next result, omitted here, can be easily obtained by following the classical strategy in optimal control of considering the limit of a minimizing sequence for the right-hand side of \eqref{value function}, exploring also Proposition~\ref{prop-adm-closed} and the bound on the value function from Proposition~\ref{prop-T-psi} in order to bound exit times and costs of the trajectories in the minimizing sequence.

\begin{proposition}
\label{PropExistOpt}
Consider the optimal control problem $\OCP(X, \Gamma, k, g)$ and assume that \ref{Hypo-X-SigmaCompact}--\ref{HypoOCP-k-Bound} and \ref{HypoOCP-g-compatible} are satisfied. Then, for every $(t_0, x_0) \in \mathbb R_+ \times X$, there exists $\gamma \in \Opt(k, g, t_0, x_0)$.
\end{proposition}

Another important classical property of optimal trajectories is that the restriction of an optimal trajectory is still optimal.

\begin{proposition}
\label{PropRestrictionIsOptimal}
Consider the optimal control problem $\OCP(X, \Gamma, k, g)$ and assume that \ref{Hypo-X-SigmaCompact}--\ref{Hypo-g}, \ref{HypoOCP-k-Bound}, and \ref{HypoOCP-g-compatible} are satisfied. Let $(t_0, x_0) \in \mathbb R_+ \times X$ and $\gamma_0 \in \Opt(k, g, t_0, x_0)$. Then, for every $t_1 \in [t_0, +\infty)$, denoting $x_1 = \gamma_0(t_1)$, the function $\gamma_1\colon \mathbb R_+ \to X$ defined by $\gamma_1(t) = x_1$ for $t \leq t_1$ and $\gamma_1(t) = \gamma_0(t)$ for $t \geq t_1$ satisfies $\gamma_1 \in \Opt(k, g, t_1, x_1)$.
\end{proposition}

\begin{proof}
If $t_1 \geq \tau(t_0, \gamma_0)$, then $x_1 \in \Gamma$ and $\gamma_1$ is the constant trajectory remaining at $x_1$ at all times. In particular, $\tau(t_1, \gamma_1) = 0$, $G(t_1, \gamma_1) = g(x_1)$, and the optimality of $\gamma_1$ follows from the fact that, if $\widetilde\gamma_1 \in \Adm(k)$ is such that $\widetilde\gamma_1(t_1) = x_1$ and $\tau(t_1, \widetilde\gamma_1) < +\infty$, then by Lemma~\ref{LemmTimePlusGIncreases}, we have $g(x_1) \leq \tau(t_1, \widetilde\gamma_1) + G(t_1, \widetilde\gamma_1)$.

Let us now consider the case $t_1 < \tau(t_0, \gamma_0)$, in which case, by definition, we have $\gamma_0(t) \notin \Gamma$ for $t \in [t_0, t_1]$, and thus $\tau(t_1, \gamma_0) = \tau(t_1, \gamma_1)$, $t_1 + \tau(t_1, \gamma_1) = t_0 + \tau(t_0, \gamma_0)$, and $G(t_1, \gamma_1) = G(t_0, \gamma_0)$. Let $\widetilde\gamma_1 \in \Adm(k)$ be such that $\widetilde\gamma_1(t_1) = x_1$. Define $\widetilde\gamma_0\colon \mathbb R_+ \to X$ by $\widetilde\gamma_0(t) = \gamma_0(t)$ for $t \in [0, t_1]$ and $\widetilde\gamma_0(t) = \widetilde\gamma_1(t)$ for $t \geq t_1$. Then $\widetilde\gamma_0 \in \Adm(k)$ and $\widetilde\gamma_0(t_0) = x_0$, showing that $\tau(t_0, \gamma_0) + G(t_0, \gamma_0) \leq \tau(t_0, \widetilde\gamma_0) + G(t_0, \widetilde\gamma_0)$. By definition of $\widetilde\gamma_0$, we have $\tau(t_1, \widetilde\gamma_0) = \tau(t_1, \widetilde\gamma_1)$ and $\widetilde\gamma_0(t) \notin \Gamma$ for every $t \in [0, t_1]$, and thus $t_1 + \tau(t_1, \widetilde\gamma_1) = t_0 + \tau(t_0, \widetilde\gamma_0)$ and $G(t_1, \widetilde\gamma_1) = G(t_0, \widetilde\gamma_0)$. We thus deduce that
\begin{align*}
\tau(t_1, \gamma_1) + G(t_1, \gamma_1) & = t_0 - t_1 + \tau(t_0, \gamma_0) + G(t_0, \gamma_0) \\
& \leq t_0 - t_1 + \tau(t_0, \widetilde\gamma_0) + G(t_0, \widetilde\gamma_0) \\
& = \tau(t_1, \widetilde\gamma_1) + G(t_1, \widetilde\gamma_1),
\end{align*}
proving that $\gamma_1 \in \Opt(k, g, t_1, x_1)$, as required.
\end{proof}

Our next result provides the dynamic programming principle for $\OCP(X, \Gamma, k, g)$.

\begin{proposition}\label{prop-DPP}
Consider the optimal control problem $\OCP(X, \Gamma, k, g)$ and assume that \ref{Hypo-X-SigmaCompact}--\ref{HypoOCP-k-Bound} and \ref{HypoOCP-g-compatible} are satisfied. Then, for every $(t_0, x_0) \in \mathbb R_+ \times X$ and $\gamma \in \Adm(k)$ with $\gamma(t_0) = x_0$, we have
\begin{equation}
\label{eq-DPP}
\varphi(t_0 + h, \gamma(t_0 + h)) + h \geq \varphi(t_0, x_0), \qquad \text{ for every } h \geq 0,
\end{equation}
with equality for every $h \in [0, \tau(t_0, \gamma)]$ if $\gamma \in \Opt(k, g, t_0, x_0)$. Moreover, if $\gamma$ is constant on $[0, t_0]$ and on $[t_0 + \tau(t_0, \gamma), +\infty)$ and if equality holds in \eqref{eq-DPP} for every $h \in [0, \tau(t_0, \gamma)]$ with $h < +\infty$, then $\tau(t_0, \gamma) < +\infty$ and $\gamma \in \Opt(k, g, t_0, x_0)$.
\end{proposition}

\begin{proof}
Let $(t_0, x_0) \in \mathbb R_+ \times X$ and $\gamma \in \Adm(k)$ with $\gamma(t_0) = x_0$. Take $h \geq 0$. Let $\widehat\gamma \in \Opt(k, g, t_0 + h, \gamma(t_0 + h))$ and define $\widetilde\gamma\colon \mathbb R_+ \to X$ by $\widetilde\gamma(t) = \gamma(t)$ for $t \in [0, t_0 + h]$ and $\widetilde\gamma(t) = \widehat\gamma(t)$ for $t \geq t_0 + h$. By definition of $\widetilde\gamma$ and $\varphi$, we have
\begin{align}
\varphi(t_0, x_0) & \leq \tau(t_0, \widetilde\gamma) + G(t_0, \widetilde\gamma), \label{eq-prf-DPP-1} \\
\varphi(t_0 + h, \gamma(t_0 + h)) & = \tau(t_0 + h, \widetilde\gamma) + G(t_0 + h, \widetilde\gamma). \notag
\intertext{In addition, by Remark~\ref{remk-tau}, we have}
t_0 + \tau(t_0, \widetilde\gamma) & \leq t_0 + h + \tau(t_0 + h, \widetilde\gamma), \label{eq-prf-DPP-2}
\intertext{and thus, by Lemma~\ref{LemmTimePlusGIncreases},}
t_0 + \tau(t_0, \widetilde\gamma) + G(t_0, \widetilde\gamma) & \leq t_0 + h + \tau(t_0 + h, \widetilde\gamma) + G(t_0 + h, \widetilde\gamma). \label{eq-prf-DPP-3}
\end{align}
Combining the above inequalities, we obtain \eqref{eq-DPP}.

In addition, if $\gamma \in \Opt(k, g, t_0, x_0)$, we have equality in \eqref{eq-prf-DPP-1} and, if $h \in [0, \tau(t_0, \gamma)]$, then we have equality in \eqref{eq-prf-DPP-2} thanks to Remark~\ref{remk-tau}, implying that we have equality in \eqref{eq-prf-DPP-3} since we apply Lemma~\ref{LemmTimePlusGIncreases} with equal times, and thus we also have equality in \eqref{eq-DPP}.

Assume now that $\gamma$ is constant on $[0, t_0]$ and on $[t_0 + \tau(t_0, \gamma), +\infty)$ and that equality holds in \eqref{eq-DPP} for every $h \in [0, \tau(t_0, \gamma)]$ with $h < +\infty$. In particular, since $\varphi$ takes nonnegative and finite values, for every such $h$, we have $h = \varphi(t_0, x_0) - \varphi(t_0 + h, \gamma(t_0 + h)) \leq \varphi(t_0, x_0)$, proving that $[0, \tau(t_0, \gamma)] \subset [0, \varphi(t_0, x_0)]$ and thus $\tau(t_0, \gamma) < +\infty$. For $h = \tau(t_0, \gamma)$, we have $\gamma(t_0 + h) \in \Gamma$, thus $\varphi(t_0 + h, \gamma(t_0 + h)) = g(\gamma(t_0 + h))$ and 
hence \eqref{eq-DPP} reads
\[
\varphi(t_0, x_0) = \tau(t_0, \gamma) + g(\gamma(t_0 + h)) = \tau(t_0, \gamma) + G(t_0, \gamma),
\]
yielding the conclusion.
\end{proof}

Our next result deals with Lipschitz continuity of the value function $\varphi$. This kind of result is classical for optimal control problems with free final time (see, e.g., \cite[Proposition~8.2.5]{Cannarsa2004Semiconcave} for a proof in the autonomous case), and a complete proof for the nonautonomous optimal control problem $\OCP(X, \Gamma, k, g)$ was given in \cite[Propositions~4.2 and 4.3]{Mazanti2019Minimal} in the case where $X$ is compact. However, that reference uses the stronger assumption that $k \in \Lip(\mathbb R_+ \times X, \mathbb R_+)$. When the optimal control problem does not have state constraints, this assumption can be relaxed to \ref{HypoOCP-k-Lip} by first showing Lipschitz continuity of $\varphi$ with respect to $x$, which can be done by adapting the classical proof of \cite[Proposition~8.2.5]{Cannarsa2004Semiconcave}, and then using the dynamic programming principle to deduce Lipschitz continuity also with respect to $t$. This strategy was described in \cite[Proposition~3.8]{Dweik2020Sharp} and carried out in details in \cite[Lemma~4.6 and Proposition~4.7]{Sadeghi2022Multi} when $X = \mathbb R^d$, however those proofs rely on the absence of state constraints and cannot be easily generalized to optimal control problems with state constraints. For that reason, we present here another proof, used in \cite[Lemma~3.2 and Proposition~3.3]{Sadeghi2022Nonsmooth} in the case where $X$ is the closure of a nonempty bounded open subset of $\mathbb R^d$. This proof is inspired by that of \cite[Propositions~4.2 and 4.3]{Mazanti2019Minimal} but uses a technique introduced in the proof of \cite[Proposition~3.9]{Dweik2020Sharp} in order to replace the assumption $k \in \Lip(\mathbb R_+ \times X, \mathbb R_+)$ by the weaker assumption \ref{HypoOCP-k-Lip}. As a first step, we prove Lipschitz continuity of $\varphi$ in space for fixed time.

\begin{lemma}
\label{lemm-varphi-Lipschitz}
Consider the optimal control problem $\OCP(X, \Gamma, k, g)$ and assume that \ref{Hypo-X-SigmaCompact}--\ref{HypoOCP-g-compatible} are satisfied. Then, for every $R > 0$, there exists $C > 0$ depending only on $\mathbf 0$, $\Gamma$, $g$, $D$, $K_{\min}$, $K_{\max}$, and $R$ such that, for every $t_0 \in \mathbb R_+$ and $x_0,\, x_1 \in \overline B_X(\mathbf 0, R)$, we have
\begin{equation}
\label{eq-varphi-Lipschitz-in-x}
\abs{\varphi(t_0, x_0) - \varphi(t_0, x_1)} \leq C \dist(x_0, x_1).
\end{equation}
\end{lemma}

\begin{proof}
It suffices to show that, for every $R > 0$, there exists $C > 0$ such that, for every $t_0 \in \mathbb R_+$ and $x_0, x_1 \in \overline B_X(\mathbf 0, R)$, we have
\begin{equation}
\label{eq-varphi-one-side-Lipschitz}
\varphi(t_0, x_1) - \varphi(t_0, x_0) \leq C \dist(x_0, x_1),
\end{equation}
since, in this case, \eqref{eq-varphi-Lipschitz-in-x} can be deduced by exchanging the role of $x_0$ and $x_1$.

Let $R > 0$, $t_0 \in \mathbb R_+$, and $x_0, x_1 \in \overline B_X(\mathbf 0, R)$. Let $\gamma_0 \in \Opt(k, g, t_0, x_0)$ and $t_0^\ast = t_0 + \tau(t_0, \gamma)$ be the time at which $\gamma_0$ arrives at the target set $\Gamma$, and set $x_0^\ast = \gamma_0(t_0^\ast) \in \Gamma$. Since $\gamma_0 \in \Opt(k, g, t_0, x_0)$, we have $\varphi(t_0, x_0) = t_0^\ast - t_0 + g(x_0^\ast)$.

Let $D > 0$ be the constant from \ref{Hypo-X-dist}, $T$ and $\psi$ be the functions from Proposition~\ref{prop-T-psi}, and $L_{\psi(R)}$ be the constant obtained from \ref{HypoOCP-k-Lip} applied to $\psi(R)$, which we denote simply by $L$ in the sequel. Note that, by Proposition~\ref{prop-T-psi}, we have $\gamma_0(t) \in \overline B_X(\mathbf 0, \psi(R))$ for every $t \geq 0$.

Applying \ref{Hypo-X-dist} to $x_0, x_1$ and renormalizing the speed of the curve whose existence is asserted in that hypothesis, we obtain the existence of $t_1 \geq t_0$ and a Lipschitz continuous curve $\sigma\colon [t_0, t_1] \to X$ such that $\sigma(t_0) = x_1$, $\sigma(t_1) = x_0$, $\abs{\dot\sigma}(t) = K_{\min}$ for almost every $t \in [t_0, t_1]$, and $t_1 - t_0 \leq \frac{D \dist(x_0, x_1)}{K_{\min}}$.

Let us now define $\phi\colon [t_1, +\infty) \to [t_0, +\infty)$ as a solution of the problem
\begin{equation}
\label{eq-EDO-phi}
\left\{
\begin{aligned}
\dot\phi(t) & = \frac{k(t, \gamma_0(\phi(t)))}{k(\phi(t), \gamma_0(\phi(t)))} & \quad & \text{ for } t \geq t_1, \\
\phi(t_1) & = t_0.
\end{aligned}
\right.
\end{equation}
Note that, since $(t, s) \mapsto \frac{k(t, \gamma_0(s))}{k(s, \gamma_0(s))}$ is continuous (but not necessarily Lipschitz continuous in its second argument), a solution $\phi$ to the above problem exists and is of class $\mathcal C^1$ (but it may not be unique). Moreover, $\dot\phi(t) \in \left[\frac{K_{\min}}{K_{\max}}, \frac{K_{\max}}{K_{\min}}\right]$ for every $t \in [t_1, +\infty)$, which implies that $\phi\colon [t_1, +\infty) \to [t_0, +\infty)$ is increasing and surjective, and hence invertible, and both $\phi$ and $\phi^{-1}$ are Lispchitz continuous, with Lipschitz constant $\frac{K_{\max}}{K_{\min}}$. We define $\sigma_1\colon [t_1, +\infty) \to X$ by $\sigma_1(t) = \gamma_0(\phi(t))$, which, by construction, satisfies $\sigma_1(t_1) = x_0$, $\sigma_1 \in \Lip(X)$, and $\abs{\dot\sigma_1}(t) = \abs{\dot\gamma_0}(\phi(t)) \dot\phi(t) \leq k(t, \sigma_1(t))$ for a.e.\ $t \in [t_1, +\infty)$. We define $t_1^\ast = \phi^{-1}(t_0^\ast)$ and remark that $\sigma_1(t_1^\ast) = \gamma_0(t_0^\ast) = x_0^\ast \in \Gamma$.

Finally, we define $\gamma_1 \in \Lip(X)$ by $\gamma_1(t) = x_1$ for $t \in [0, t_0]$, $\gamma_1(t) = \sigma(t)$ for $t \in [t_0, t_1]$, and $\gamma_1(t) = \sigma_1(t)$ for $t \in [t_1, +\infty)$. By construction, we have $\gamma_1 \in \Adm(k)$ and $\tau(t_0, \gamma_1) \leq t_1^\ast - t_0$.

Note that, for every $t \in [t_1, +\infty)$, we have $k(\phi(t), \gamma_0(\phi(t))) \dot\phi(t) = k(t, \gamma_0(\phi(t)))$ and thus, by integrating this identity and performing a change of variables, we deduce that, for every $t \geq t_1$,
\begin{equation}
\label{eq-interesting-identitiy}
\int_{t_0}^{\phi(t)} k(s, \gamma_0(s)) \diff s = \int_{t_1}^t k(s, \gamma_0(\phi(s))) \diff s.
\end{equation}
Let $H\colon [t_0, +\infty) \to [0, +\infty)$ be defined for $\theta \geq t_0$ by
\[
H(\theta) = \int_{t_0}^\theta k(s, \gamma_0(s)) \diff s.
\]
Then $H$ is differentiable, with $\dot H(\theta) = k(\theta, \gamma_0(\theta)) \in [K_{\min}, K_{\max}]$ for every $\theta \geq t_0$. In particular, $H$ is $K_{\max}$-Lipschitz continuous, invertible, and its inverse is $\frac{1}{K_{\min}}$-Lipschitz continuous. Moreover, using \eqref{eq-interesting-identitiy}, we have, for every $t \geq t_1$, that
\[
H(\phi(t)) = \int_{t_1}^t k(s, \gamma_0(\phi(s))) \diff s, \qquad H(t) = \int_{t_0}^t k(s, \gamma_0(s)) \diff s.
\]
Hence, for every $t \geq t_1$, we have
\begin{equation}
\label{eq-estim-phit-minus-t}
\begin{aligned}
\abs{\phi(t) - t} & = \abs*{H^{-1}\left(\int_{t_1}^t k(s, \gamma_0(\phi(s))) \diff s\right) - H^{-1}\left(\int_{t_0}^t k(s, \gamma_0(s)) \diff s\right)} \\
& \leq \frac{1}{K_{\min}} \abs*{\int_{t_1}^t k(s, \gamma_0(\phi(s))) \diff s - \int_{t_0}^t k(s, \gamma_0(s)) \diff s} \\
& \leq \frac{1}{K_{\min}} \int_{t_1}^t \abs*{k(s, \gamma_0(\phi(s))) - k(s, \gamma_0(s))}\diff s + \frac{1}{K_{\min}}\int_{t_0}^{t_1} k(s, \gamma_0(s)) \diff s \\
& \leq \frac{L K_{\max}}{K_{\min}} \int_{t_1}^t \abs*{\phi(s) - s} \diff s + \frac{K_{\max}}{K_{\min}} (t_1 - t_0).
\end{aligned}
\end{equation}
Thus, by Grönwall's inequality, we have, for every $t \geq t_1$,
\[
\abs{\phi(t) - t} \leq (t_1 - t_0) \frac{K_{\max}}{K_{\min}} e^{\frac{L K_{\max}}{K_{\min}} (t - t_1)},
\]
which yields, for $t = t_1^\ast$, that
\[
\abs*{t_1^\ast - t_0^\ast} \leq (t_1 - t_0) \frac{K_{\max}}{K_{\min}} e^{\frac{L K_{\max}}{K_{\min}} (t_1^\ast - t_1)}.
\]
Note that $0 \leq t_1^\ast - t_1 = \phi^{-1}(t_0^\ast) - \phi^{-1}(t_0) \leq \frac{K_{\max}}{K_{\min}} (t_0^\ast - t_0) \leq \frac{K_{\max}}{K_{\min}} \varphi(t_0, x_0) \leq \frac{T(R) K_{\max}}{K_{\min}}$. Recalling that $0 \leq t_1 - t_0 \leq \frac{D \dist(x_1, x_0)}{K_{\min}}$, we deduce that
\[
t_1^\ast - t_0^\ast \leq C \dist(x_0, x_1),
\]
with $C = \frac{D K_{\max}}{K_{\min}^2} \exp\left(\frac{T(R) L K_{\max}^2}{K_{\min}^2}\right)$. This implies \eqref{eq-varphi-one-side-Lipschitz} since, using Lemma~\ref{LemmTimePlusGIncreases}, we have $\varphi(t_0, x_1) \leq \tau(t_0, \gamma_1) + G(t_0, \gamma_1) \leq t_1^\ast - t_0 + g(x_0^\ast)$ and $\varphi(t_0, x_0) = t_0^\ast - t_0 + g(x_0^\ast)$.
\end{proof}

Now, exactly as in \cite[Proposition~4.7]{Sadeghi2022Multi} and \cite[Proposition~3.3]{Sadeghi2022Nonsmooth}, we can deduce Lipschitz continuity of $\varphi$ by using Lemma~\ref{lemm-varphi-Lipschitz} and the dynamic programming principle from Proposition~\ref{prop-DPP}. The proof of our next result, omitted here, can be obtained through a straightforward adaptation of \cite[Proposition~3.3]{Sadeghi2022Nonsmooth} to our more general setting.

\begin{proposition}
\label{prop-varphi-Lipschitz}
Consider the optimal control problem $\OCP(X, \Gamma, k, g)$ and assume that \ref{Hypo-X-SigmaCompact}--\ref{HypoOCP-g-compatible} are satisfied. Then, for every $R > 0$, there exists $M > 0$ depending only on $\mathbf 0$, $\Gamma$, $g$, $D$, $K_{\min}$, $K_{\max}$, and $R$ such that, for every $(t_0, x_0), (t_1, x_1) \in \mathbb R_+ \times \overline B_X(\mathbf 0, R)$, we have
\[
\abs{\varphi(t_0, x_0) - \varphi(t_1, x_1)} \leq M \left(\abs{t_0 - t_1} + \dist(x_0, x_1)\right).
\] 
\end{proposition}

We conclude these preliminary results on $\OCP(X, \Gamma, k, g)$ with the following property of the value function $\varphi$. This property was first established in \cite[Proposition~4.4]{Mazanti2019Minimal} for compact $X$ and with $g = 0$, but using the stronger assumption that $k \in \Lip(\mathbb R_+ \times X, \mathbb R_+)$. We present here instead a proof based on the sharper result \cite[Proposition~3.9]{Dweik2020Sharp}, which uses only an assumption on the Lipschitz behavior of $k$ similar to the weaker assumption \ref{HypoOCP-k-Lip} but was shown in that reference only in the case where $X$ is the closure of a nonempty bounded open subset of $\mathbb R^d$ and without state constraints.

\begin{proposition}
\label{PropMonotoneOptimalTime}
Consider the optimal control problem $\OCP(X, \Gamma, k, g)$ and assume that \ref{Hypo-X-SigmaCompact}--\ref{HypoOCP-g-compatible} are satisfied. Then, for every $R > 0$, there exists $c > 0$ such that, for every $x \in \overline B_X(\mathbf 0, R)$ and $t_0, t_1 \in \mathbb R_+$ with $t_0 \neq t_1$, we have
\begin{equation}
\label{DtTauQGreaterThanMinusOne}
\frac{\varphi(t_0, x) - \varphi(t_1, x)}{t_0 - t_1} \geq c-1.
\end{equation}
\end{proposition}

\begin{proof}
Let $R > 0$, $T$ and $\psi$ be the functions from Proposition~\ref{prop-T-psi}, and $L_{\psi(R)}$ be the constant from \ref{HypoOCP-k-Lip} applied with $R$ replaced by $\psi(R)$. We denote $L_{\psi(R)}$ simply by $L$ in the sequel for simplicity. Let $x \in \overline B_X(\mathbf 0, R)$ and $t_0, t_1 \in \mathbb R_+$ with $t_0 \neq t_1$ and suppose, without loss of generality, that $t_1 < t_0$.

Let $\gamma_0 \in \Opt(k, g, t_0, x)$ and $x_0 = \gamma_0(t_0 + \tau(t_0, \gamma)) \in \Gamma$ and recall that, by Proposition~\ref{prop-T-psi}, we have $\gamma_0(t) \in \overline B_X(\mathbf 0, \psi(R))$ for every $t \in \mathbb R_+$. Similarly to the proof of Lemma~\ref{lemm-varphi-Lipschitz}, let $\phi\colon [t_1, +\infty)\allowbreak \to [t_0, +\infty)$ be a function satisfying \eqref{eq-EDO-phi}. Note that, with no loss of generality, we may assume that $\phi(t) \geq t$ for every $t \in [t_1, +\infty)$. Indeed, $\phi(t_1) = t_0 > t_1$ and, if there exists $t > t_1$ such that $\phi(t) < t$, then, setting $t_2 = \inf\{t \in [t_1, +\infty) \suchthat \phi(t) = t\}$, we define $\widetilde\phi\colon [t_1, +\infty) \to [t_0, +\infty)$ by $\widetilde\phi(t) = \phi(t)$ for $t \in [t_1, t_2]$ and $\widetilde\phi(t) = t$ for $t \geq t_2$. It is immediate to verify that $\widetilde\phi$ is a solution of \eqref{eq-EDO-phi} with $\widetilde\phi(t) \geq t$ for every $t \in [t_1, +\infty)$. We thus assume in the sequel that $\phi(t) \geq t$ for every $t \in [t_1, +\infty)$.

Define $\gamma_1\colon \mathbb R_+ \to X$ by $\gamma_1(t) = x$ for $t \in [0, t_1]$ and $\gamma_1(t) = \gamma_0(\phi(t))$ for all $t \geq t_1$. By construction, we have $\gamma_1 \in \Adm(k)$, $t_0 + \tau(t_0, \gamma_0) = \phi(t_1 + \tau(t_1, \gamma_1))$, and $t_0 + \varphi(t_0, x) = t_0 + \tau(t_0, \gamma_0) + g(x_0) = \phi(t_1 + \tau(t_1, \gamma_1)) + g(x_0)$. For $i \in \{1, 2\}$, we set $t_i^\ast = t_i + \tau(t_i, \gamma_i)$, so that $\phi(t_1^\ast) = t_0^\ast$. In particular, $t_0^\ast \geq t_1^\ast$.

Similarly to the proof of Lemma~\ref{lemm-varphi-Lipschitz}, we have the identity
\[
\int_{\phi(t)}^{t_0^\ast} k(s, \gamma_0(s)) \diff s = \int_t^{t_1^\ast} k(s, \gamma_0(\phi(s))) \diff s \qquad \text{ for every } t \in [t_1, t_1^\ast].
\]
Defining this time $H\colon [0, t_0^\ast] \to \mathbb R_+$ for $\theta \in [0, t_0^\ast]$ by
\[
H(\theta) = \int_{\theta}^{t_0^\ast} k(s, \gamma_0(s)) \diff s,
\]
we have
\begin{alignat*}{2}
H(\phi(t)) & = \int_t^{t_1^\ast} k(s, \gamma_0(\phi(s))) \diff s, & \qquad & \text{ for } t \in [t_1, t_1^\ast], \\
H(t) & = \int_{t}^{t_0^\ast} k(s, \gamma_0(s)) \diff s, & & \text{ for } t \in [0, t_0^\ast],
\end{alignat*}
and thus, proceeding as in the estimate \eqref{eq-estim-phit-minus-t} from the proof of Lemma~\ref{lemm-varphi-Lipschitz}, we deduce that, for every $t \in [t_1, t_1^\ast]$,
\[
0 \leq \phi(t) - t \leq \frac{L K_{\max}}{K_{\min}} \int_t^{t_1^\ast} (\phi(s) - s) \diff s + \frac{K_{\max}}{K_{\min}} (t_0^\ast - t_1^\ast).
\]
Defining $\lambda(r) = \phi(t_1^\ast - r) - t_1^\ast + r$ for $r \in [0, \tau(t_1, \gamma_1)]$, rewriting the above inequality in terms of $\lambda$, and using Grönwall's inequality, we deduce that, for every $t \in [t_1, t_1^\ast]$, we have
\[
0 \leq \phi(t) - t \leq \frac{K_{\max}}{K_{\min}} (t_0^\ast - t_1^\ast) e^{\frac{L K_{\max}}{K_{\min}} (t_1^\ast - t)},
\]
and thus
\begin{equation}
\label{eq-estim-t0-minus-t1}
t_0 - t_1 \leq \frac{K_{\max}}{K_{\min}} (t_0^\ast - t_1^\ast) e^{\frac{L K_{\max}}{K_{\min}} \tau(t_1, \gamma_1)}.
\end{equation}

Using Proposition~\ref{prop-T-psi} and the fact that $\phi^{-1}$ is $\frac{K_{\max}}{K_{\min}}$-Lipschitz continuous, we get the estimate
\begin{equation*}
\tau(t_1, \gamma_1) = \phi^{-1}(t_0^\ast) - t_1 = \phi^{-1}(t_0^\ast) - \phi^{-1}(t_0) \leq \frac{K_{\max}}{K_{\min}} \tau(t_0, \gamma_0) \leq \frac{K_{\max}}{K_{\min}} \varphi(t_0, x) \leq \frac{K_{\max}}{K_{\min}} T(R),
\end{equation*}
and thus, setting $c = \frac{K_{\min}}{K_{\max}} \exp\left(-\frac{T(R) L K_{\max}^2}{K_{\min}^2}\right) > 0$, we deduce from \eqref{eq-estim-t0-minus-t1} that $c(t_0 - t_1) \leq t_0^\ast - t_1^\ast$, that is,
\[
(c-1)(t_0 - t_1) \leq \tau(t_0, \gamma_0) - \tau(t_1, \gamma_1),
\]
and the conclusion follows by noticing that $\varphi(t_0, x) = \tau(t_0, \gamma_0) + g(x_0)$ and $\varphi(t_1, x) \leq \tau(t_1, \gamma_1) + g(x_0)$.
\end{proof}

\subsection{Properties of the mean field game}

Using the properties of the optimal control problem $\OCP(X, \Gamma, k, g)$, we now establish, in this section, some properties of the mean field game $\MFG(X, \Gamma, K, g, m_0)$. Since $X$ is not assumed to be compact, we need here results that are finer than those from \cite{Mazanti2019Minimal, Dweik2020Sharp, Sadeghi2021Characterization, Sadeghi2022Nonsmooth}, and, for this reason, we follow the approach from \cite{Sadeghi2022Multi}. However, some adaptations are needed, since \cite{Sadeghi2022Multi} only deals with the case of the Euclidean space $X = \mathbb R^d$, without considering state constraints, and it also only considers the minimal-time problem, while we also have the exit cost $g$ here.

Recall that, according to Definition~\ref{DefiEquilibriumMFG}, (weak or strong) equilibria of $\MFG(X,\allowbreak \Gamma,\allowbreak K,\allowbreak g,\allowbreak m_0)$ are described in terms of a measure $Q \in \mathcal P(\mathcal C_X)$. Given such a measure $Q$, we shall consider the optimal control problem $\OCP(X, \Gamma, k_{Q}, g)$, with $k_{Q}$ given by $k_{Q}(t, x) = K({e_t}_{\#} Q, x)$ for $(t, x) \in \mathbb R_+ \times X$. We will denote the value function of $\OCP(X,\allowbreak \Gamma,\allowbreak k_{Q},\allowbreak g)$ by $\varphi_{Q}$, and we omit $Q$ from the notation of both $k_{Q}$ and $\varphi_{Q}$ when it is clear from the context.

Given $Q \in \mathcal C_X$, we denote in the sequel by $\OOpt(Q)$ the set
\begin{equation}
\label{eq-OOpt}
\OOpt(Q) = \bigcup_{x \in X} \Opt(Q, g, 0, x),
\end{equation}
i.e., $\OOpt(Q)$ is the set of optimal trajectories for $\OCP(X, \Gamma, k_Q, g)$ starting at time $0$. Note that $Q$ is a weak equilibrium if and only if $Q(\OOpt(Q)) = 1$, and it is a strong equilibrium if and only if $\supp(Q) \subset \OOpt(Q)$.

We start with the following preliminary result, whose proof is straightforward.

\begin{lemma}
\label{lemm-etQ-continuous}
\mbox{}
\begin{enumerate}
\item For every $Q \in \mathcal P(\mathcal C_X)$, the function $\mathbb R_+ \ni t \mapsto {e_t}_{\#} Q \in \mathcal P(X)$ is continuous.
\item For every $t \in \mathbb R_+$, the function $\mathcal P(\mathcal C_X) \ni Q \mapsto {e_t}_{\#} Q \in \mathcal P(X)$ is continuous.
\end{enumerate}
\end{lemma}

As a consequence, we immediately obtain the following result.

\begin{corollary}
\label{coro-k-from-K}
Let $K\colon \mathcal P(X) \times X \to \mathbb R_+$ be a function satisfying \ref{HypoMFG-K-Bound} and \ref{HypoMFG-K-Lip}, $Q \in \mathcal P(\mathcal C_X)$, and $k_Q\colon \mathbb R_+ \times X \to \mathbb R_+$ be defined for $(t, x) \in \mathbb R_+ \times X$ by $k_Q(t, x) = K({e_t}_{\#}Q, x)$. Then $k_Q$ satisfies \ref{HypoOCP-k-Bound} and \ref{HypoOCP-k-Lip} with the same constants $K_{\min}$, $K_{\max}$, and $L_R$ as $K$.
\end{corollary}

The next result, which is a consequence of Proposition~\ref{prop-T-psi}, states an important a priori property of equilibria, which will allow us in the sequel to restrict our search for equilibria to a compact space. This result was already stated, without proof, in \cite[Lemma~5.3]{Sadeghi2022Multi}, and we provide its proof below for completeness.

\begin{proposition}
\label{prop-equilibrium-psi}
Consider the mean field game $\MFG(X, \Gamma, K, g, m_0)$ and assume that \ref{Hypo-X-SigmaCompact}--\ref{Hypo-X-dist}, \ref{HypoMFG-K-Bound}, and \ref{HypoMFG-g-compatible} are satisfied. Then there exists a nondecreasing function with linear growth $\psi\colon \mathbb R_+ \to \mathbb R_+$, depending only on $\mathbf 0$, $\Gamma$, $g$, $D$, $K_{\min}$, and $K_{\max}$, such that, for every weak equilibrium $Q \in \mathcal P(\mathcal C_X)$ of $\MFG(X, \Gamma, K, g, m_0)$ and $R > 0$, we have
\[
Q\Bigl(\Lip_{K_{\max}}\bigl(\overline B_X(\mathbf 0, \psi(R))\bigr)\Bigr) \geq m_0\bigl(\overline B_X(\mathbf 0, R)\bigr).
\]
In particular, we have ${e_t}_{\#} Q\bigl(\overline B_X(\mathbf 0, \psi(R))\bigr) \geq m_0\bigl(\overline B_X(\mathbf 0, R)\bigr)$ for every $t \in \mathbb R_+$.
\end{proposition}

\begin{proof}
Let $Q \in \mathcal P(\mathcal C_X)$ be a weak equilibrium of $\MFG(X, \Gamma, K, g, m_0)$, $k\colon \mathbb R_+ \times X \to \mathbb R_+$ be defined for $(t, x) \in \mathbb R_+ \times X$ by $k(t, x) = K({e_t}_{\#} Q, x)$, and $\psi$ be the function obtained by applying Proposition~\ref{prop-T-psi} to $\OCP(X, \Gamma, k, g)$. Recalling Corollary~\ref{coro-k-from-K}, we observe that $\psi$ only depends on $\mathbf 0$, $\Gamma$, $g$, $D$, $K_{\min}$, and $K_{\max}$, and, in particular, it is independent of $Q$. By Proposition~\ref{prop-T-psi}, we have
\[
\{\gamma \in \mathcal C_X \suchthat \gamma(0) \in \overline B_X(\mathbf 0, R)\} \cap \OOpt(Q) \subset \Lip_{K_{\max}}\bigl(\overline B_X(\mathbf 0, \psi(R))\bigr) \cap \OOpt(Q),
\]
where $\OOpt(Q)$ is the set defined in \eqref{eq-OOpt}. Recall that, $Q$ being a weak equilibrium, we have $Q(\OOpt(Q)) = 1$, and thus the above inclusion implies that
\begin{align*}
Q\Bigl(\Lip_{K_{\max}}\bigl(\overline B_X(\mathbf 0, \psi(R))\bigr)\Bigr) & = Q\Bigl(\Lip_{K_{\max}}\bigl(\overline B_X(\mathbf 0, \psi(R))\bigr) \cap \OOpt(Q)\Bigr) \\
& \geq Q\bigl(\{\gamma \in \mathcal C_X \suchthat \gamma(0) \in \overline B_X(\mathbf 0, R)\} \cap \OOpt(Q)\bigr) \\
& = Q\bigl(\{\gamma \in \mathcal C_X \suchthat \gamma(0) \in \overline B_X(\mathbf 0, R)\}\bigr) = m_0\bigl(\overline B_X(\mathbf 0, R)\bigr),
\end{align*}
since ${e_0}_{\#} Q = m_0$.

Finally, the last part of the conclusion follows as an immediate consequence of the inclusion $\Lip_{K_{\max}}\bigl(\overline B_X(\mathbf 0, \psi(R))\bigr) \subset {e_t}^{-1}\bigl(\overline B_X(\mathbf 0, \psi(R))\bigr)$.
\end{proof}

It follows from Proposition~\ref{prop-equilibrium-psi} that it suffices to look for equilibria of the mean field game $\MFG(X, \Gamma, K, g, m_0)$ in the set
\begin{multline}
\label{eq-defi-Q}
\mathfrak{Q} = \Bigl\{Q \in \mathcal{P}(\mathcal C_X) \mathrel{\Big\vert} {e_0}_{\#} Q = m_0 \text{ and } \\ \forall R > 0, Q\Bigl(\Lip_{K_{\max}}\bigl(\overline B_X(\mathbf 0, \psi(R))\bigr)\Bigr) \geq m_0\bigl(\overline B_X(\mathbf 0, R)\bigr)\Bigr\},
\end{multline}
where $\psi$ is the function from the statement of Proposition~\ref{prop-equilibrium-psi}. The next result provides elementary properties of $\mathfrak Q$. Its proof, omitted here, is identical to that of \cite[Lemma~5.4]{Sadeghi2022Multi}, since the fact that \cite{Sadeghi2022Multi} considers $X = \mathbb R^d$ and $g = 0$ plays no particular role in the proof.

\begin{proposition}\label{PropQNonemptyConvexCompact}
Consider the mean field game $\MFG(X, \Gamma, K, g, m_0)$, assume that \ref{Hypo-X-SigmaCompact}--\ref{Hypo-X-dist}, \ref{HypoMFG-K-Bound}, and \ref{HypoMFG-g-compatible} are satisfied, and let $\mathfrak Q$ be the set defined in \eqref{eq-defi-Q}. Then $\mathfrak Q$ is nonempty, convex, and compact.
\end{proposition}

We now prove an additional continuity property of the value function, stating that it is continuous when it is regarded as a function not only of $(t, x)$ but also of the measure $Q$. This property was stated in \cite[Proposition~5.2]{Mazanti2019Minimal} and \cite[Lemma~5.2]{Sadeghi2022Multi} in the case $g = 0$, but none of these proofs can be immediately applied to our setting: \cite[Proposition~5.2]{Mazanti2019Minimal} makes the additional assumption of Lipschitz behavior of $K$ with respect to its first variable, deducing Lipschitz behavior of $\varphi$ with respect to $Q$, which is not necessarily the case in our context, while \cite[Lemma~5.2]{Sadeghi2022Multi} makes use of the Euclidean structure of $\mathbb R^d$ and the absence of state constraints in some steps of the proof, exploiting also deeply the fact that $g = 0$. The proof we present here makes use of a different strategy, its ideas being inspired by a combination of those of \cite[Proposition~5.2]{Mazanti2019Minimal} and \cite[Proposition~3.9]{Dweik2020Sharp}, and it is reminiscent of the strategy of proof used in \cite[Lemma~4.5]{Dweik2020Sharp}.

\begin{proposition}
\label{PropValueFunctionContinuous}
Consider the mean field game $\MFG(X, \Gamma, K, g, m_0)$ and assume that \ref{Hypo-X-SigmaCompact}--\ref{Hypo-X-dist} and \ref{HypoMFG-K-Bound}--\ref{HypoMFG-g-compatible} are satisfied. Then $(t, x, Q) \mapsto \varphi_Q(t, x)$ is continuous on $\mathbb R_+ \times X \times \mathcal P(\mathcal C_X)$.
\end{proposition}

\begin{proof}
Let $(t_n, x_n, Q_n)_{n \in \mathbb N}$ be a sequence in $\mathbb R_+ \times X \times \mathcal P(\mathcal C_X)$ converging to some $(t_\ast, x_\ast, Q_\ast)$. For $n \in \mathbb N$ and $(t, x) \in \mathbb R_+ \times X$, define $k_n(t, x) = K({e_t}_{\#} Q_n, x)$ and $k_\ast(t, x) = K({e_t}_{\#} Q_\ast, x)$. Note that, by Lemma~\ref{lemm-etQ-continuous}, we have that $k_n(t, x) \to k_\ast(t, x)$ for every $(t, x) \in \mathbb R_+ \times X$. For simplicity of notation, we write $\varphi_n$ and $\varphi_\ast$ for $\varphi_{Q_n}$ and $\varphi_{Q_{\ast}}$, respectively.

Since $x_n \to x_\ast$ as $n \to +\infty$, there exists $R > 0$ such that $x_n \in \overline B_X(\mathbf 0, R)$ for every $n \in \mathbb N$ and $x_\ast \in \overline B_X(\mathbf 0, R)$. Let $T$ and $\psi$ be the functions obtained when one applies Proposition~\ref{prop-T-psi} to the functions $k_n$, $n \in \mathbb N$, and $k_\ast$; note that, by Proposition~\ref{prop-T-psi}, these functions do not depend on $n$, since $k_n$, $n \in \mathbb N$, and $k_\ast$ all satisfy \ref{HypoOCP-k-Bound} and \ref{HypoOCP-k-Lip} with the same constants $K_{\min}$, $K_{\max}$, and $L_R$, thanks to Corollary~\ref{coro-k-from-K}. Similarly, Proposition~\ref{prop-varphi-Lipschitz} applied to $\OCP(X, \Gamma, k_n, g)$, $n \in \mathbb N$, and to $\OCP(X, \Gamma, k_\ast, g)$ yields a constant $M > 0$ independent of $n$ such that $\varphi_n$, $n \in \mathbb N$, and $\varphi_\ast$ are all $M$-Lipschitz continuous in $\mathbb R_+ \times \overline B_X(\mathbf 0, R)$. In particular, we have
\[
\abs{\varphi_n(t_n, x_n) - \varphi_\ast(t_\ast, x_\ast)} \leq M (\abs{t_n - t_\ast} + \dist(x_n, x_\ast)) + \abs{\varphi_n(t_\ast, x_\ast) - \varphi_\ast(t_\ast, x_\ast)}.
\]
Thus, to conclude the proof, it suffices to show that $\varphi_n(t_\ast, x_\ast) \to \varphi_\ast(t_\ast, x_\ast)$ as $n \to +\infty$. We split this proof into two parts, showing first that $\varphi_\ast(t_\ast, x_\ast) \leq \liminf_{n \to +\infty} \varphi_n(t_\ast, x_\ast)$ and then that $\limsup_{n \to +\infty} \varphi_n(t_\ast, x_\ast) \leq \varphi_\ast(t_\ast, x_\ast)$.

Let us first show that $\varphi_\ast(t_\ast, x_\ast) \leq \liminf_{n \to +\infty} \varphi_n(t_\ast, x_\ast)$. Let $\underline{\varphi} = \liminf_{n \to +\infty} \allowbreak \varphi_n(t_\ast, x_\ast)$ and consider a subsequence of $(\varphi_n(t_\ast, x_\ast))_{n \in \mathbb N}$ converging to $\underline\varphi$, which we still denote by $(\varphi_n(t_\ast, x_\ast))_{n \in \mathbb N}$ for simplicity. For $n \in \mathbb N$, let $\gamma_n \in \Opt(Q_n, g, t_\ast, x_\ast)$ and recall that, by Proposition~\ref{prop-T-psi}, we have $\gamma_n(t) \in \overline B_X(\mathbf 0, \psi(R))$ for every $n \in \mathbb N$ and $t \in \mathbb R_+$. Since in addition $\gamma_n$ is $K_{\max}$-Lipschitz continuous for every $n \in \mathbb N$, by Arzelà--Ascoli theorem, up to extracting a subsequence, there exists $\gamma_\ast \in \Lip_{K_{\max}}(X)$ such that $\gamma_n \to \gamma_\ast$ as $n \to +\infty$ in the topology of $\mathcal C_X$ and $\gamma_\ast(t) \in \overline B_X(\mathbf 0, \psi(R))$ for every $t \in \mathbb R_+$.

Using the facts that $k_n(t, x) \to k_\ast(t, x)$ for every $(t, x) \in \mathbb R_+ \times X$ and that $k_n$, $n \in \mathbb N$, and $k_\ast$ are Lipschitz continuous in their second argument in $\overline B_X(\mathbf 0, \psi(R))$, uniformly with respect to the first argument in $\mathbb R_+$ and to $n \in \mathbb N$, we deduce, proceeding as in the proof of Proposition~\ref{prop-adm-closed}, that $\gamma_\ast \in \Adm(k_\ast)$. In addition, since $\gamma_n(t_\ast) = x_\ast$ for every $n \in \mathbb N$, we have $\gamma_\ast(t_\ast) = x_\ast$. Thus, $\varphi_\ast(t_\ast, x_\ast) \leq \tau(t_\ast, \gamma_\ast) + G(t_\ast, \gamma_\ast)$ and, by Lemma~\ref{lemm-tau-plus-G-lsc}, we have $\tau(t_\ast, \gamma_\ast) + G(t_\ast, \gamma_\ast) \leq \liminf_{n \to +\infty} \tau(t_\ast, \gamma_n) + G(t_\ast, \gamma_n) = \underline\varphi$, yielding the conclusion.

Let us now prove that $\limsup_{n \to +\infty} \varphi_n(t_\ast, x_\ast) \leq \varphi_\ast(t_\ast, x_\ast)$. Let $\gamma_\ast \in \Opt(Q_\ast, g, t_\ast, x_\ast)$ and recall that, by Proposition~\ref{prop-T-psi}, we have $\gamma_\ast(t) \in \overline B_X(\mathbf 0, \psi(R))$ for every $t \in \mathbb R_+$. In the sequel, we let $L$ denote the constant $L_{\psi(R)}$ from \ref{HypoMFG-K-Lip}. For $n \in \mathbb N$, let $\phi_n\colon [t_\ast, +\infty) \to [t_\ast, +\infty)$ be a function of class $\mathcal C^1$ satisfying
\[
\left\{
\begin{aligned}
\dot\phi_n(t) & = \frac{k_n(t, \gamma_\ast(\phi_n(t)))}{k_\ast(\phi_n(t), \gamma_\ast(\phi_n(t)))} & \quad & \text{ for } t \geq t_\ast, \\
\phi_n(t_\ast) & = t_\ast.
\end{aligned}
\right.
\]
Let $\gamma_n\colon \mathbb R_+ \to X$ be defined by $\gamma_n(t) = x_\ast$ for $t \in [0, t_\ast]$ and $\gamma_n(t) = \gamma_\ast(\phi_n(t))$ for $t \in [t_\ast, +\infty)$. By construction, we have $\gamma_n \in \Adm(Q_n)$ and $t_\ast + \tau(t_\ast, \gamma_\ast) = \phi_n(t_\ast + \tau(t_\ast, \gamma_n))$.

We now prove that $\phi_n(t) \to t$ as $n \to +\infty$, uniformly for $t$ on a compact interval of $[t_\ast, +\infty)$. To do so, we proceed as in the proof of Lemma~\ref{lemm-varphi-Lipschitz}: for every $t \in [t_\ast, +\infty)$, we have
\[
\int_{t_\ast}^{\phi_n(t)} k_\ast(s, \gamma_\ast(s)) \diff s = \int_{t_\ast}^t k_n(s, \gamma_\ast(\phi_n(s))) \diff s.
\]
Defining $H\colon [t_\ast, +\infty) \to [t_\ast, +\infty)$ for $\theta \in [t_\ast, +\infty)$ by
\[
H(\theta) = \int_{t_\ast}^\theta k_\ast(s, \gamma_\ast(s)) \diff s,
\]
we then have the estimate, for every $t \in [t_\ast, +\infty)$,
\begin{align*}
\abs*{\phi_n(t) - t} & = \abs*{H^{-1}\left(\int_{t_\ast}^t k_n(s, \gamma_\ast(\phi_n(s))) \diff s\right) - H^{-1}\left(\int_{t_\ast}^t k_\ast(s, \gamma_\ast(s)) \diff s\right)} \\
& \leq \frac{1}{K_{\min}} \int_{t_\ast}^t \abs*{k_n(s, \gamma_\ast(\phi_n(s))) - k_\ast(s, \gamma_\ast(s))} \diff s \\
& \leq \frac{1}{K_{\min}} \int_{t_\ast}^t \abs*{k_n(s, \gamma_\ast(\phi_n(s))) - k_n(s, \gamma_\ast(s))} \diff s \\
& \hphantom{{} \leq {}} + \frac{1}{K_{\min}} \int_{t_\ast}^t \abs*{k_n(s, \gamma_\ast(s)) - k_\ast(s, \gamma_\ast(s))} \diff s \\
& \leq \frac{L K_{\max}}{K_{\min}} \int_{t_\ast}^t \abs*{\phi_n(s) - s} \diff s + \frac{1}{K_{\min}} \int_{t_\ast}^t \abs*{k_n(s, \gamma_\ast(s)) - k_\ast(s, \gamma_\ast(s))} \diff s,
\end{align*}
which yields, by Grönwall's inequality, that, for every $t \in [t_\ast, +\infty)$,
\[
\abs{\phi_n(t) - t} \leq \frac{1}{K_{\min}} e^{\frac{L K_{\max}}{K_{\min}} (t - t_\ast)} \int_{t_\ast}^t \abs*{k_n(s, \gamma_\ast(s)) - k_\ast(s, \gamma_\ast(s))} \diff s.
\]
Since $k_n(s, \gamma_\ast(s)) \to k_\ast(s, \gamma_\ast(s))$ for every $s \in [t_\ast, t]$ and $k_n$, $n \in \mathbb N$, and $k_\ast$ are positive functions upper bounded by $K_{\max}$, it follows from Lebesgue's dominated convergence theorem that the integral in the right-hand side of the above inequality converges to $0$ as $n \to +\infty$, and such a convergence can be made uniformly in $t$ for $t$ on a compact interval of $[t_\ast, +\infty)$. Hence, $\phi_n(t) \to t$ as $n \to +\infty$, uniformly in $t$ for $t$ on a compact interval of $[t_\ast, +\infty)$, as required. Such a convergence also implies that $\phi_n^{-1}(t) \to t$ as $n \to +\infty$ uniformly in $t$ for $t$ on a compact interval of $[t_\ast, +\infty)$, since $\abs*{\phi_n^{-1}(t) - t} \leq \frac{K_{\max}}{K_{\min}} \abs{t - \phi_n(t)}$, using the fact that $\phi_n^{-1}$ is $\frac{K_{\max}}{K_{\min}}$-Lipschitz continuous.

Now, recalling that $t_\ast + \tau(t_\ast, \gamma_\ast) = \phi_n(t_\ast + \tau(t_\ast, \gamma_n))$, we have $\tau(t_\ast, \gamma_n) = \phi_n^{-1}(t_\ast + \tau(t_\ast, \gamma_\ast)) - t_\ast \to \tau(t_\ast, \gamma_\ast)$ as $n \to +\infty$. Thus, $G(t_\ast, \gamma_n) = g(\gamma_n(t_\ast + \tau(t_\ast, \gamma_n))) \to G(t_\ast, \gamma_\ast)$ as $n \to +\infty$. As a conclusion,
\[
\limsup_{n \to +\infty} \varphi_n(t_\ast, x_\ast) \leq \lim_{n \to +\infty} \left[\tau(t_\ast, \gamma_n) + G(t_\ast, \gamma_n)\right] = \tau(t_\ast, \gamma_\ast) + G(t_\ast, \gamma_\ast) = \varphi_\ast(t_\ast, x_\ast). \qedhere
\]
\end{proof}

One of the consequences of Proposition~\ref{PropValueFunctionContinuous} is the following property of the graph of the set-valued map $\OOpt$ defined in \eqref{eq-OOpt}.

\begin{proposition}
\label{prop-OOpt-closed-graph}
Consider the mean field game $\MFG(X, \Gamma, K, g, m_0)$, assume that \ref{Hypo-X-SigmaCompact}--\ref{Hypo-X-dist} and \ref{HypoMFG-K-Bound}--\ref{HypoMFG-g-compatible} are satisfied, and let $\OOpt\colon \mathcal P(\mathcal C_X) \rightrightarrows \mathcal C_X$ be the set-valued map defined for $Q \in \mathcal P(\mathcal C_X)$ by \eqref{eq-OOpt}. Then $\OOpt$ has closed graph.
\end{proposition}

\begin{proof}
Let $(Q_n)_{n \in \mathbb N}$ be a sequence in $\mathcal P(\mathcal C_X)$ with $Q_n \to Q$ as $n \to +\infty$ for some $Q \in \mathcal P(\mathcal C_X)$ and $(\gamma_n)_{n \in \mathbb N}$ be a sequence in $\mathcal C_X$ with $\gamma_n \in \OOpt(Q_n)$ for every $n \in \mathbb N$ and $\gamma_n \to \gamma$ as $n \to +\infty$ for some $\gamma \in \mathcal C_X$. For every $n \in \mathbb N$, since $\gamma_n \in \OOpt(Q_n) \subset \Adm(Q_n)$, we have in particular that $\gamma_n \in \Lip_{K_{\max}}(X)$, and thus $\gamma \in \Lip_{K_{\max}}(X)$. Proceeding as in the proof of Propositions~\ref{prop-adm-closed} and \ref{PropValueFunctionContinuous}, we obtain that $\gamma \in \Adm(Q)$. Using Lemma~\ref{lemm-tau-plus-G-lsc} and Proposition~\ref{PropValueFunctionContinuous}, we have that
\[
\tau(0, \gamma) + G(0, \gamma) \leq \liminf_{n \to +\infty} \tau(0, \gamma_n) + G(0, \gamma_n) = \liminf_{n \to +\infty} \varphi_{Q_n}(0, \gamma_n(0)) = \varphi_Q(0, \gamma(0)),
\]
and thus, up to modifying $\gamma$ on the interval $[\tau(0, \gamma), +\infty)$ in order for it to be constant there, we deduce that $\gamma \in \Opt(Q, g, 0, \gamma(0))$, thus $\gamma \in \OOpt(Q)$, as required.
\end{proof}

As a consequence of Proposition~\ref{prop-OOpt-closed-graph}, we deduce that the set-valued map $(Q, x) \mapsto \Opt(Q, g, 0, x)$ also has closed graph. This property was shown in \cite[Lemma~4.5]{Dweik2020Sharp} for the mean field game model considered in that reference, using a direct proof relying neither on the closedness of the graph of $\OOpt$ nor on the continuity of $(Q, t, x) \mapsto \varphi_Q(t, x)$. Thanks to our previous results, we provide here a much simpler proof.

\begin{corollary}
\label{coro-Opt-closed-graph}
Consider the mean field game $\MFG(X, \Gamma, K, g, m_0)$ and assume that \ref{Hypo-X-SigmaCompact}--\ref{Hypo-X-dist} and \ref{HypoMFG-K-Bound}--\ref{HypoMFG-g-compatible} are satisfied. Then the set-valued map $\Opt(\cdot, g, 0, \cdot)\colon\allowbreak \mathcal P(\mathcal C_X) \times X \rightrightarrows \mathcal C_X$ has closed graph.
\end{corollary}

\begin{proof}
Note that, if $(Q_n, x_n, \gamma_n)_{n \in \mathbb N}$ is a sequence in $\mathcal P(\mathcal C_X) \times X \times \mathcal C_X$ converging to some element $(Q, x, \gamma) \in \mathcal P(\mathcal C_X) \times X \times \mathcal C_X$ and with $\gamma_n \in \Opt(Q_n, g, 0, x_n)$ for every $n \in \mathbb N$, then $\gamma_n \in \OOpt(Q_n)$ for every $n \in \mathbb N$ and, by Proposition~\ref{prop-OOpt-closed-graph}, we have $\gamma \in \OOpt(Q)$. Since $\gamma(0) = \lim_{n \to +\infty} \gamma_n(0) = \lim_{n \to +\infty} x_n = x$, we deduce that $\gamma \in \Opt(Q, g, 0, x)$, as required.
\end{proof}

Another important consequence of Proposition~\ref{prop-OOpt-closed-graph} is that the notions of weak and strong equilibria coincide, a fact already stated in \cite[Remark~4.6]{Dweik2020Sharp} and \cite[Proposition~3.7]{Sadeghi2022Nonsmooth} for the models considered in those references. We deduce it here as a trivial consequence of Proposition~\ref{prop-OOpt-closed-graph}.

\begin{corollary}
\label{coro-strong-iff-weak}
Consider the mean field game $\MFG(X, \Gamma, K, g, m_0)$ and assume that \ref{Hypo-X-SigmaCompact}--\ref{Hypo-X-dist} and \ref{HypoMFG-K-Bound}--\ref{HypoMFG-g-compatible} are satisfied. Let $Q \in \mathcal P(\mathcal C_X)$. Then $Q$ is a weak equilibrium of $\MFG(X, \Gamma, K, g, m_0)$ if and only if it is a strong equilibrium of $\MFG(X, \Gamma, K, g, m_0)$.
\end{corollary}

\begin{proof}
Recall that, as stated in Remark~\ref{remk-equilibria-strong-weak}, any strong equilibrium of $\MFG(X, \Gamma, K, g, m_0)$ is also a weak equilibrium. To prove the converse implication, let $Q \in \mathcal P(\mathcal C_X)$ be a weak equilibrium of $\MFG(X, \Gamma, K, g, m_0)$. Take $\gamma \in \supp(Q)$. Then, by definition of support and using the fact that $Q(\OOpt(Q)) = 1$, there exists a sequence $(\gamma_n)_{n \in \mathbb N}$ in $\OOpt(Q)$ such that $\gamma_n \to \gamma$ as $n \to +\infty$. It follows from Proposition~\ref{prop-OOpt-closed-graph} that $\OOpt(Q)$ is closed, hence $\gamma \in \OOpt(Q)$. Thus $\supp(Q) \subset \OOpt(Q)$, concluding the proof.
\end{proof}

Finally, another consequence that we can get from from Proposition~\ref{prop-OOpt-closed-graph} is that, given $Q \in \mathcal P(\mathcal C_X)$, one can find a Borel-measurable function that, with each $x \in X$, associates an optimal trajectory $\gamma \in \Opt(Q, g, 0, x)$. A similar result was stated in \cite[Proposition~5.3]{Mazanti2019Minimal}, but with a weaker conclusion, asserting measurability only with respect to a $\sigma$-algebra larger than that of Borel in $X$. We provide here more detailed arguments that allow one to obtain Borel measurability, based on the ideas given in the proof of \cite[Lemma~4.7]{Dweik2020Sharp} and adapting them to the case of a space $X$ not necessarily compact.

\begin{corollary}
\label{coro-measurable-selection}
Consider the mean field game $\MFG(X, \Gamma, K, g, m_0)$ and assume that \ref{Hypo-X-SigmaCompact}--\ref{Hypo-X-dist} and \ref{HypoMFG-K-Bound}--\ref{HypoMFG-g-compatible} are satisfied. For every $Q \in \mathcal P(\mathcal C_X)$, there exists a Borel-measurable function $\Phi\colon X \to \OOpt(Q)$ such that $\Phi(x) \in \Opt(Q, g, 0, x)$ for every $x \in X$.
\end{corollary}

\begin{proof}
Take $Q \in \mathcal P(\mathcal C_X)$ and let $k\colon \mathbb R_+ \times X \to \mathbb R_+$ be defined by $k(t, x) = K({e_t}_{\#} Q, x)$. Let $\psi\colon \mathbb R_+ \to \mathbb R_+$ be the function obtained by applying Proposition~\ref{prop-T-psi} to $\OCP(X, \Gamma, k, g)$.

Let $\pmb\Phi\colon X \rightrightarrows \OOpt(Q)$ be the set-valued function defined for $x \in X$ by $\pmb\Phi(x) = \Opt(Q, g, 0, x)$ and note that, as an immediate consequence of Corollary~\ref{coro-Opt-closed-graph}, $\pmb\Phi$ has closed graph. In addition, by Proposition~\ref{PropExistOpt}, $\pmb\Phi(x) \neq \emptyset$ for every $x \in X$.

We claim that $\pmb\Phi$ is upper semicontinuous. Indeed, if it were not the case, there would exist $x \in X$, an open set $U$ containing $\pmb\Phi(x)$ in $\mathcal C_X$, a sequence $(x_n)_{n \in \mathbb N}$ in $X$ converging to $x$, and a sequence $(\gamma_n)_{n \in \mathbb N}$ in $\mathcal C_X$ such that $\gamma_n \in \pmb\Phi(x_n) = \Opt(Q, g, 0, x_n)$ for every $n \in \mathbb N$ and $\gamma_n \notin U$. Note that, in particular, $\gamma_n \in \Lip_{K_{\max}}(X)$.

Let $R > 0$ be such that $x_n \in \overline B_X(\mathbf 0, R)$ for every $n \in \mathbb N$ and $x \in \overline B_X(\mathbf 0, R)$. Then, by Proposition~\ref{prop-T-psi}, we have $\gamma_n(t) \in \overline B_X(\mathbf 0, \psi(R))$ for every $n \in \mathbb N$ and $t \geq 0$, i.e., $(\gamma_n)_{n \in \mathbb N}$ is a sequence in $\Lip_{K_{\max}}(\overline B_X(\mathbf 0, \psi(R)))$. Since the latter set is compact, up to extracting a subsequence, there exists $\gamma \in \Lip_{K_{\max}}(\overline B_X(\mathbf 0, \psi(R)))$ such that $\gamma_n \to \gamma$ as $n \to +\infty$. Hence, the sequence $(Q, x_n, \gamma_n)_{n \in \mathbb N}$ converges to $(Q, x, \gamma)$ and satisfies $\gamma_n \in \Opt(Q, g, 0, x_n)$ for every $n \in \mathbb N$, showing, by Corollary~\ref{coro-Opt-closed-graph}, that $\gamma \in \Opt(Q, g, 0, x) \subset U$. Since $\gamma_n \to \gamma$ as $n \to +\infty$ and $U$ is open, we deduce that $\gamma_n \in U$ for $n$ large enough, which contradicts the fact that $\gamma_n \notin U$ for every $n \in \mathbb N$. This contradiction establishes the fact that $\pmb\Phi$ is upper semicontinuous.

Since $\pmb\Phi$ is upper semicontinuous, by \cite[Proposition~1.4.4]{Aubin2009Set}, the set $\pmb\Phi^{-1}(A) = \{x \in X \suchthat \pmb\Phi(x) \cap A \neq \emptyset\}$ is closed for every closed subset $A$ of $\mathcal C_X$. Hence, by \cite[Proposition~III.11]{Castaing1977Convex}, the set $\pmb\Phi^{-1}(B)$ is open for every open subset $B$ of $\mathcal C_X$. Thus, $\pmb\Phi$ is measurable, and we obtain the desired function $\Phi$ as a Borel-measurable selection of $\pmb\Phi$, which exists thanks to \cite[Theorem~8.1.3]{Aubin2009Set}.
\end{proof}

\section{Main results}
\label{sec:main}

Thanks to the preliminary results from Section~\ref{sec-prelim}, we are now in position to state and prove the main results of this paper. In all the discussions in this section, we assume that \ref{Hypo-X-SigmaCompact}--\ref{Hypo-X-dist} and \ref{HypoMFG-K-Bound}--\ref{HypoMFG-g-compatible} are satisfied and, in particular, by Corollary~\ref{coro-strong-iff-weak}, the notions of strong and weak equilibria for $\MFG(X, \Gamma, K, g, m_0)$ coincide. For that reason, we refer in the sequel to a strong or weak equilibrium simply as \emph{equilibrium}.

\subsection{Existence of equilibria}
\label{sec:existence}

As in \cite{Mazanti2019Minimal, Dweik2020Sharp, Sadeghi2021Characterization, Sadeghi2022Multi, Sadeghi2022Nonsmooth}, our strategy to prove existence of equilibria consists in recasting the notion of equilibrium in terms of fixed points of a set-valued map and then applying to the latter a suitable fixed-point theorem which, in our case, will be Kakutani fixed-point theorem (see, e.g., \cite[\S~7, Theorem~8.6]{Granas2003Fixed}).

Let $\OOpt$ and $\mathfrak Q$ be defined as in \eqref{eq-OOpt} and \eqref{eq-defi-Q}, respectively. We define the set-valued map $F\colon \mathfrak Q \rightrightarrows \mathfrak Q$ by setting, for $Q \in \mathfrak Q$,
\begin{equation}
\label{eq-defi-F}
F(Q) = \left\{\widetilde Q \in \mathfrak Q \suchthat \widetilde Q(\OOpt(Q)) = 1\right\}.
\end{equation}
Clearly, $Q$ is an equilibrium of $\MFG(X, \Gamma, K, g, m_0)$ if and only if $Q \in F(Q)$. To apply Kakutani fixed-point theorem to $F$, we first prove the following properties of $F$. These properties were shown in \cite[Lemmas~5.3 and~5.5]{Mazanti2019Minimal}, \cite[Lemma~4.7]{Dweik2020Sharp}, and \cite[Lemma~5.6]{Sadeghi2022Multi} for the mean field game models considered in those papers, and we proceed here by following the line of the proof of \cite[Lemma~5.6]{Sadeghi2022Multi}, adapting arguments to our more general setting when needed.

\begin{lemma}
\label{lemm-F}
Consider the mean field game $\MFG(X, \Gamma, K, g, m_0)$, assume that \ref{Hypo-X-SigmaCompact}--\ref{Hypo-X-dist} and \ref{HypoMFG-K-Bound}--\ref{HypoMFG-g-compatible} are satisfied, and let $F\colon \mathfrak Q \rightrightarrows \mathfrak Q$ be defined as in \eqref{eq-defi-F}. Then $F$ is upper semicontinuous and, for every $Q \in \mathfrak Q$, $F(Q)$ is nonempty, compact, and convex.
\end{lemma}

\begin{proof}
Let $Q \in \mathfrak Q$. Clearly, by \eqref{eq-defi-F}, the set $F(Q)$ is convex. To prove that it is nonempty, let $\Phi\colon X \to \OOpt(Q)$ be the map from Corollary~\ref{coro-measurable-selection} and set $\widetilde Q = \Phi_{\#} m_0$. By construction, $\widetilde Q \in \mathcal P(\mathcal C_X)$, ${e_0}_{\#} \widetilde Q = m_0$, and $\widetilde Q(\OOpt(Q)) = 1$. Using the latter fact and proceeding as in the proof of Proposition~\ref{prop-equilibrium-psi}, we deduce that $\widetilde Q\Bigl(\Lip_{K_{\max}}\bigl(\overline B_X(\mathbf 0, \psi(R))\bigr)\Bigr) \geq m_0\bigl(\overline B_X(\mathbf 0, R)\bigr)$ for every $R > 0$, and thus $\widetilde Q \in \mathfrak Q$. Hence $\widetilde Q \in F(Q)$, so that $F(Q)$ is nonempty.

Let us now show that $F(Q)$ is compact. Since $F(Q) \subset \mathfrak Q$ and $\mathfrak Q$ is compact by Proposition~\ref{PropQNonemptyConvexCompact}, it suffices to prove that $F(Q)$ is closed. Let $(\widetilde Q_n)_{n \in \mathbb N}$ be a sequence in $F(Q)$ converging as $n \to +\infty$ to some $\widetilde Q \in \mathcal P(\mathcal C_X)$. Since $\widetilde Q_n \in F(Q) \subset \mathfrak Q$ and $\mathfrak Q$ is closed, we deduce that $\widetilde Q \in \mathfrak Q$. Note also that, as a consequence of Proposition~\ref{prop-OOpt-closed-graph}, $\OOpt(Q)$ is closed. Hence, by the Portmanteau theorem, we have
\[
\widetilde Q(\OOpt(Q)) \geq \limsup_{n \to +\infty} \widetilde Q_n(\OOpt(Q)) = 1,
\]
which proves that $\widetilde Q(\OOpt(Q)) = 1$, and thus $\widetilde Q \in F(Q)$. Hence $F(Q)$ is closed.

Let us finally show that $F$ is upper semicontinuous. Since $\mathfrak Q$ is compact, is suffices to show, by \cite[Proposition~1.4.8]{Aubin2009Set}, that the graph of $F$ is closed. Let $(Q_n, \widetilde Q_n)_{n \in \mathbb N}$ be a sequence in $\mathfrak{Q} \times \mathfrak Q$ with $(Q_n, \widetilde Q_n) \to (Q, \widetilde Q)$ for some $(Q, \widetilde Q) \in \mathfrak{Q} \times \mathfrak Q$ and such that $\widetilde Q_n \in F(Q_n)$ for every $n \in \mathbb N$.

For each $n \in \mathbb N$, since $\widetilde Q_n \in F(Q_n)$, we have $\widetilde Q_n(\OOpt(Q_n)) = 1$ and, since $\widetilde Q_n \in \mathfrak Q$, we also have that $\widetilde Q_n\Bigl(\Lip_{K_{\max}}\bigl(\overline B_X(\mathbf 0, \psi(R))\bigr)\Bigr) \geq m_0\bigl(\overline B_X(\mathbf 0, R)\bigr)$ for every $R > 0$, where $\psi$ is the function from Proposition~\ref{prop-equilibrium-psi}.

For every $\varepsilon \in (0, 1)$, let $V_{\varepsilon} = \{\gamma \in \mathcal C_X \suchthat \dist_{\mathcal C_X}(\gamma, \OOpt(Q)) \leq \varepsilon\}$. Let $R_0 > 0$ be such that $m_0\bigl(\overline B_X(\mathbf 0, R_0)\bigr) \geq 1 - \varepsilon$. We claim that there exists $N_{\varepsilon} \in \mathbb N$ such that, for all $n \in \mathbb N$ with $n \geq N_\varepsilon$, we have
\begin{equation}
\label{eq-LipKmax-cap-OOpt-subset-Vepsilon}
\Lip_{K_{\max}}\bigl(\overline B_X(\mathbf 0, \psi(R_0))\bigr) \cap \OOpt(Q_n) \subset V_\varepsilon.
\end{equation}
Indeed, if it were not the case, then there would exist a subsequence of $(Q_n)_{n \in \mathbb N}$, still denoted by $(Q_n)_{n \in \mathbb N}$ for simplicity, and a sequence $(\gamma_n)_{n \in \mathbb N}$ in $\Lip_{K_{\max}}\bigl(\overline B_X(\mathbf 0, \psi(R_0))\bigr)$ such that, for every $n \in \mathbb N$, we have $\gamma_n \in \OOpt(Q_n)$ and $\gamma_n \notin V_\varepsilon$. Since $\Lip_{K_{\max}}\bigl(\overline B_X(\mathbf 0, \psi(R_0))\bigr)$ is compact, up to extracting a subsequence, there exists $\gamma \in \Lip_{K_{\max}}\bigl(\overline B_X(\mathbf 0, \psi(R_0))\bigr)$ such that $\gamma_n \to \gamma$ as $n \to +\infty$. By Proposition~\ref{prop-OOpt-closed-graph}, we also deduce that $\gamma \in \OOpt(Q)$. Since $\gamma_n \notin V_\varepsilon$ for every $n \in \mathbb N$, we have $\dist_{\mathcal C_X}(\gamma_n, \OOpt(Q)) > \varepsilon$ for every $n \in \mathbb N$, thus $\dist_{\mathcal C_X}(\gamma, \OOpt(Q)) \geq \varepsilon$, which contradicts the fact that $\gamma \in \OOpt(Q)$. This contradiction establishes that there exists $N_{\varepsilon} \in \mathbb N$ such that \eqref{eq-LipKmax-cap-OOpt-subset-Vepsilon} holds true for every $n \in \mathbb N$ with $n \geq N_\varepsilon$.

Since $\widetilde Q_n \in \mathfrak Q$ and $\widetilde Q_n(\OOpt(Q_n)) = 1$, we obtain from \eqref{eq-LipKmax-cap-OOpt-subset-Vepsilon} that, for every $n \in \mathbb N$ with $n \geq N_\varepsilon$, we have
\[
\widetilde{Q}_{n}(V_{\varepsilon}) \geq \widetilde{Q}_{n}\Bigl(\Lip_{K_{\max}}\bigl(\overline B_X(\mathbf 0, \psi(R_0))\bigr) \cap \OOpt(Q_n)\Bigr) \geq m_0\bigl(\overline B_X(\mathbf 0, R_0)\bigr) \geq 1 - \varepsilon.
\]
Since $\widetilde{Q}_n \to \widetilde{Q}$ and $V_{\varepsilon}$ is closed, we have from the Portmanteau theorem that $\widetilde{Q}(V_{\varepsilon}) \geq \limsup_{n\to +\infty} \widetilde{Q}_{n}(V_{\varepsilon}) \geq 1 - \varepsilon$. On the other hand, since $\OOpt(Q)$ is closed and $(V_\varepsilon)_{\varepsilon \in (0, 1)}$ is a nondecreasing family of sets with $\bigcap_{\varepsilon \in (0, 1)} V_{\varepsilon} = \OOpt(Q)$, we conclude that $\widetilde Q(\OOpt(Q)) = \lim_{\varepsilon \to 0} \widetilde{Q}(V_{\varepsilon}) = 1$. Hence $\widetilde{Q} \in F(Q)$, concluding the proof that the graph of $F$ is closed.
\end{proof}

Using Lemma~\ref{lemm-F}, we are finally in position to state and prove our main result on the existence of equilibria for $\MFG(X, \Gamma, K, g, m_0)$. Recall that, by Corollary~\ref{coro-strong-iff-weak}, weak and strong equilibria of $\MFG(X, \Gamma, K, g, m_0)$ coincide under our assumptions.

\begin{theorem}
\label{thm-exist-equilibrium}
Consider the mean field game $\MFG(X, \Gamma, K, g, m_0)$ and assume that \ref{Hypo-X-SigmaCompact}--\ref{Hypo-X-dist} and \ref{HypoMFG-K-Bound}--\ref{HypoMFG-g-compatible} are satisfied. Then there exists an equilibrium $Q \in \mathcal P(\mathcal C_X)$ of $\MFG(X, \Gamma, K, g, m_0)$.
\end{theorem}

\begin{proof}
Let $F\colon \mathfrak Q \rightrightarrows \mathfrak Q$ be the set-valued map defined by \eqref{eq-defi-F}. By Proposition~\ref{PropQNonemptyConvexCompact}, Lem\-ma~\ref{lemm-F}, and the Kakutani fixed-point theorem (see, e.g., \cite[\S~7, Theorem~8.6]{Granas2003Fixed}), $F$ admits a fixed point, which, as discussed before, is an equilibrium of $\MFG(X, \Gamma, K, g, m_0)$.
\end{proof}

\begin{remark}
\label{RemkEquilibriumNonUnique}
Given $m_0 \in \mathcal P(X)$, one may have several equilibria of $\MFG(X, \Gamma, K, g, m_0)$, as one may see from the following example taken from \cite[Remark~7.1]{Mazanti2019Minimal}. Let $X = [0, 1]$, $\Gamma = \{0, 1\}$, and $K$ and $g$ be constant functions equal, respectively, to $1$ and $0$. Assume that $m_0$ is the Dirac delta measure on the point $\frac{1}{2}$. Let $\gamma_\ell, \gamma_r \in \mathcal C_X$ be given for $t \in \mathbb R_+$ by $\gamma_\ell(t) = \max\left(\frac{1}{2} - t, 0\right)$ and $\gamma_r(t) = \min\left(\frac{1}{2} + t, 1\right)$. Then any $Q \in \mathcal P(\mathcal C_X)$ concentrated on $\gamma_\ell, \gamma_r$ (i.e., satisfying $Q(\{\gamma_\ell, \gamma_r\}) = 1$) is an equilibrium for $\MFG(X, \Gamma, K, g, m_0)$. This example can be generalized for $X = \overline\Omega$ and $\Gamma = \partial\Omega$, where $\Omega \subset \mathbb R^d$ is a nonempty bounded open set, by taking $K$ and $g$ as before and considering initial distributions $m_0 \in \mathcal P(\overline\Omega)$ concentrated on the set where the distance function $\dist(\cdot, \partial\Omega)$ is not differentiable.

For other models of mean field games, uniqueness of equilibria is typically obtained under some monotonicity assumptions on functions appearing in the cost of each player (see, e.g., \cite[Proposition~2.9 and Theorem~3.6]{CardaliaguetNotes}, \cite[Theorem~4.1]{Lasry2006JeuxI}, and \cite[Theorem~3.1]{Lasry2006JeuxII}). Typically, these monotonicity assumptions mean that players tend to avoid congested regions, and they are important for uniqueness since games in which players tend to aggregate may present several equilibria (see, e.g., \cite{Cirant2019Time}). In our setting, it is not clear whether suitable congestion-avoidance assumptions should be sufficient for obtaining uniqueness of equilibria.
\end{remark}

\subsection{Asymptotic behavior}
\label{sec:asymptotic}

Given an equilibrium $Q \in \mathcal P(\mathcal C_X)$ of a mean field game $\MFG(X, \Gamma, K, g, m_0)$, we now wish to understand the asymptotic behavior of the distribution of agents at time $t$, $m_t = {e_t}_{\#} Q$, as $t \to +\infty$. It is natural to expect that $m_t$ converges, as $t \to +\infty$, to a measure $m_\infty$ concentrated in $\Gamma$ and, in addition to prove this, the goal of this section is to provide also the convergence rate of $m_t$ to $m_\infty$ in the Wasserstein distance. This question was addressed in \cite[Section~5.2]{Sadeghi2022Multi} for the mean field game $\MFG(\mathbb R^d, \Gamma, K, 0, m_0)$ (in a multipopulation setting), and it turns out that it is not difficult to generalize the arguments from \cite{Sadeghi2022Multi} to a set $X$ satisfying \ref{Hypo-X-SigmaCompact} and to the presence of an exit cost $g$ satisfying \ref{Hypo-g}. For sake of completeness, we provide below the details of this generalization.

In order to characterize the limit of $m_t$ as $t \to +\infty$, we will need some additional notation, which we now state. Let $\mathcal C_{\lim}(X) = \{\gamma \in \mathcal C_X \suchthat \lim_{t \to +\infty} \gamma(t) \text{ exists}\}$, which is a Borel subset of $\mathcal C_X$, and define $e_{\infty}\colon \mathcal C_{\lim}(X) \to X$ by $e_\infty(\gamma) = \lim_{t \to +\infty} \gamma(t)$, which is a Borel-measurable function. Note that $\mathcal C_{\lim}(X)$ is nonempty, since it contains, for instance, all constant trajectories (and also all continuous trajectories that become constant after some time). In addition, by definition of optimal trajectories, we have $\OOpt(Q) \subset \mathcal C_{\lim}(X)$ for every $Q \in \mathcal P(\mathcal C_X)$, and thus ${e_\infty}_{\#} Q \in \mathcal P(X)$ is well-defined for every equilibrium $Q$ of a mean field game $\MFG(X, \Gamma, K, g, m_0)$.

Our main result on the asymptotic behavior of equilibria is the following.

\begin{theorem}
\label{ThmAsymp}
Consider the mean field game $\MFG(X, \Gamma, K, g, m_0)$ and assume that \ref{Hypo-X-SigmaCompact}--\ref{Hypo-X-dist} and \ref{HypoMFG-K-Bound}--\ref{HypoMFG-g-compatible} are satisfied. Let $Q \in \mathcal P(\mathcal C_X)$ be an equilibrium of $\MFG(X,\allowbreak \Gamma,\allowbreak K,\allowbreak g,\allowbreak m_0)$ and, for $t \in [0, +\infty]$, define $m_t = {e_t}_{\#} Q$. Let $\psi$ be the function obtained by applying Proposition~\ref{prop-T-psi} to $\OCP(X, \Gamma, k, g)$, where $k(t, x) = K(m_t, x)$ for $(t, x) \in \mathbb R_+ \times X$.

\begin{enumerate}
\item\label{ThmAsymp-Converg} We have $m_t \to m_\infty$ as $t \to +\infty$.

\item\label{ThmAsymp-Wp} Let $p \in [1, +\infty)$ and assume that $m_0 \in \mathcal P_p(X)$. Then, for every $t \in [0, +\infty]$, we have $m_t \in \mathcal P_p(X)$. Moreover, there exist constants $\alpha > 0$ and $t_0 \geq 0$ only depending on $\mathbf 0$, $\Gamma$, $g$, $D$, $K_{\min}$, and $K_{\max}$ such that, for all $t \geq t_0$, we have
\begin{equation}
\label{eq-cv-rate-Wp}
\mathbf W_p(m_t, m_\infty)^p \leq 2^p \int_{X \setminus \overline B_X(\mathbf 0, \alpha(t - t_0))} \psi(\dist(\mathbf 0, x))^p \diff m_0(x).
\end{equation}

\item\label{ThmAsymp-FiniteTime} Assume that $m_0$ is compactly supported. Then, for every $t \in [0, +\infty]$, $m_t$ is compactly supported and there exists $t_\ast \geq 0$ such that, for every $t \geq t_\ast$, we have
\[
m_t = m_\infty.
\]
\end{enumerate}
\end{theorem}

\begin{proof}
Let $T$ be the function obtained by applying Proposition~\ref{prop-T-psi} to $\OCP(X, \Gamma, k, g)$, and let $\alpha > 0$ and $t_0 \geq 0$ be such that $T(R) \leq \frac{R}{\alpha} + t_0$ for every $R > 0$. Note that, thanks to Corollary~\ref{coro-k-from-K}, $\alpha$ and $t_0$ can be chosen to depend only on $\mathbf 0$, $\Gamma$, $g$, $D$, $K_{\min}$, and $K_{\max}$.

To show \ref{ThmAsymp-Converg}, let $f\colon X \to \mathbb R$ be continuous and bounded. We then have, using the continuity and boundedness of $f$ and Lebesgue's dominated convergence theorem, that
\begin{multline*}
\int_{X} f(x) \diff m_t(x) = \int_{\mathcal{C}_{\lim}(X)} f(\gamma(t)) \diff Q(\gamma) \\ \xrightarrow[t \to +\infty]{} \int_{\mathcal C_{\lim}(X)} f\Bigl(\lim_{t \to +\infty} \gamma(t)\Bigr) \diff Q(\gamma) = \int_{X} f(x) \diff m_\infty(x),
\end{multline*}
yielding the required convergence.

Let us now prove \ref{ThmAsymp-Wp}. By Proposition~\ref{prop-T-psi}, for every $\gamma \in \OOpt(Q)$, we have that $\gamma(t) \in \overline B_X \Bigl(\mathbf 0, \psi\bigl(\dist(\mathbf 0, \gamma(0))\bigr)\Bigr)$ for every $t \in \mathbb R_+$, and the same is still true for $t = +\infty$ by taking the limit $t \to +\infty$. Hence, for $t \in [0, +\infty]$, we have
\begin{multline*}
\int_{X} \dist(\mathbf 0, x)^p \diff m_t(x) = \int_{\OOpt(Q)} \dist(\mathbf 0, \gamma(t))^p \diff Q(\gamma) \\ \leq \int_{\OOpt(Q)} \psi(\dist(\mathbf 0, \gamma(0)))^p \diff Q(\gamma) = \int_{X} \psi(\dist(\mathbf 0, x))^p \diff m_0(x),
\end{multline*}
where $\gamma(\infty)$ is defined for $\gamma \in \OOpt(Q)$ as $\lim_{t \to +\infty} \gamma(t)$. Since $\psi$ has linear growth, it follows that $m_t \in \mathcal P_p(X)$ for every $t \in [0, +\infty]$.

Let $t \in [t_0, +\infty)$. Note that, using the notation introduced in Section~\ref{sec:notation}, we have $(e_t, e_\infty)_{\#} Q \in \Pi(m_t, m_\infty)$ and thus, by \eqref{eq-defi-Wasserstein}, we have
\[
\mathbf{W}_p(m_t, m_\infty)^p \leq \int_{X \times X} \dist(x, y)^p \diff (e_t, e_\infty)_{\#} Q(x, y) = \int_{\OOpt(Q)} \dist(e_t(\gamma), e_{\infty}(\gamma))^p \diff Q(\gamma).
\]
If $\gamma \in \OOpt(Q)$ is such that $\dist(\mathbf 0, \gamma(0)) \leq \alpha(t - t_0)$, then, since $T(\dist(\mathbf 0, \gamma(0))) \leq t$, we have, as a consequence of Proposition~\ref{prop-T-psi}, that $\gamma(t) \in \Gamma$ and $\gamma$ is constant on $[t, +\infty)$, yielding that $e_t(\gamma) = e_\infty(\gamma)$. Thus
\begin{equation*}
\mathbf{W}_p(m_t, m_\infty)^p \leq \int_{\OOpt(Q) \cap \left\{\gamma \in \mathcal C_X \suchthat \dist(\mathbf 0, \gamma(0)) > \alpha(t - t_0)\right\}} \dist(e_t(\gamma), e_{\infty}(\gamma))^p \diff Q(\gamma).
\end{equation*}
Recalling that $e_t(\gamma) = \gamma(t) \in \overline B_X \Bigl(\mathbf 0, \psi\bigl(\dist(\mathbf 0, \gamma(0))\bigr)\Bigr)$ for every $t \in [0, +\infty]$ and $\gamma \in \OOpt(Q)$, one has $\dist(\mathbf 0, e_t(\gamma)) \leq \psi\bigl(\dist(\mathbf 0, \gamma(0))\bigr)$ and thus
\begin{align*}
\mathbf{W}_p(m_t, m_\infty)^p & \leq \text{$\displaystyle \int_{\OOpt(Q) \cap \left\{\gamma \in \mathcal C_X \suchthat \dist(\mathbf 0, \gamma(0)) > \alpha(t - t_0)\right\}}$} 2^p \psi\bigl(\text{$\dist(\mathbf 0, \gamma(0))$}\bigr)^p \diff Q(\gamma) \displaybreak[0]\\
& = 2^p \int_{X \setminus \overline B_X(\mathbf 0, \alpha(t - t_0))} \psi(\dist(\mathbf 0, x))^p \diff m_0(x),
\end{align*}
as required.

Finally, to prove \ref{ThmAsymp-FiniteTime}, let $R_0 > 0$ be such that the support of $m_0$ is included in $\overline B_X(\mathbf 0, R_0)$ and notice that, as a consequence of Proposition~\ref{prop-T-psi}, the support of $m_t$ is included in the compact set $\overline B_X(\mathbf 0, \psi(R_0))$ for every $t \in [0, +\infty]$. Letting $t_\ast = T(R_0)$, we deduce that, for every $t \geq t_\ast$ and $\gamma \in \OOpt(Q)$ with $\gamma(0) \in \overline B_X(\mathbf 0, R_0)$, we have $e_t(\gamma) = e_\infty(\gamma)$, which concludes the proof since $Q$ is supported in $\OOpt(Q) \cap \{\gamma \in \mathcal C_X \suchthat \gamma(0) \in \overline B_X(\mathbf 0, R_0)\}$.
\end{proof}

\begin{remark}
The limit $m_\infty \in \mathcal P(X)$ from Theorem~\ref{ThmAsymp}\ref{ThmAsymp-Converg} is characterized as $m_\infty = {e_\infty}_\# Q$, and hence it is uniquely determined by $Q$. However, as equilibria $Q$ of a given mean field game $\MFG(X, \Gamma, K, g, m_0)$ are not necessarily unique, $m_\infty$ is also not necessarily uniquely determined by $m_0$. Indeed, for the example of mean field game considered in Remark~\ref{RemkEquilibriumNonUnique}, $m_\infty$ can be any measure of the form $\alpha \delta_0 + (1 - \alpha) \delta_1$ for $\alpha \in [0, 1]$, where $\delta_a$ denotes the Dirac delta measure at $a$.
\end{remark}

We conclude this section by illustrating how \eqref{eq-cv-rate-Wp} can be used to obtain explicit convergence rates of $m_t$ to $m_\infty$ in the Wasserstein distance $\mathbf W_p$ in the case $X = \mathbb R^d$.

\begin{corollary}
\label{coro:asymp}
Let $d \in \mathbb N^\ast$ and assume that $\mathbb R^d$ is endowed with the Euclidean distance. Consider the mean field game $\MFG(\mathbb R^d, \Gamma, K, g, m_0)$ and assume that \ref{Hypo-Gamma}, \ref{Hypo-g}, and \ref{HypoMFG-K-Bound}--\ref{HypoMFG-g-compatible} are satisfied. Let $Q \in \mathcal P(\mathcal C_X)$ be an equilibrium of $\MFG(X,\allowbreak \Gamma,\allowbreak K,\allowbreak g,\allowbreak m_0)$ and, for $t \in [0, +\infty]$, define $m_t = {e_t}_{\#} Q$. Fix $p \in [1, +\infty)$.
\begin{enumerate}
\item\label{item:poly-decay} Assume that there exist $R_0 > 0$, $C_0 > 0$, and $\beta \in (p + d, +\infty)$ such that the restriction of $m_0$ to $\mathbb R^d \setminus \overline B_{\mathbb R^d}(0, R_0)$ is absolutely continuous with respect to the Lebesgue measure and its density, still denoted by $m_0$, satisfies $m_0(x) \leq \frac{C_0}{\abs{x}^\beta}$ for almost every $x \in \mathbb R^d \setminus \overline B_{\mathbb R^d}(0, R_0)$. Then there exist $T_\ast > 0$ (only depending on $\Gamma$, $g$, $K_{\min}$, $K_{\max}$, and $R_0$) and $C_\ast > 0$ (only depending on $\Gamma$, $g$, $K_{\min}$, $K_{\max}$, $C_0$, $d$, $p$, and $\beta$) such that, for every $t > T_\ast$, we have
\[
\mathbf W_p(m_t, m_\infty)^p \leq \frac{C_\ast}{t^{\beta - p - d}}.
\]

\item\label{item:exp-decay} Assume that there exist $R_0 > 0$, $C_0 > 0$, and $\gamma_0 > 0$ such that the restriction of $m_0$ to $\mathbb R^d \setminus \overline B_{\mathbb R^d}(0, R_0)$ is absolutely continuous with respect to the Lebesgue measure and its density, still denoted by $m_0$, satisfies $m_0(x) \leq C_0 e^{-\gamma_0 \abs{x}}$ for almost every $x \in \mathbb R^d \setminus \overline B_{\mathbb R^d}(0, R_0)$. Then there exist $T_\ast > 0$ (only depending on $\Gamma$, $g$, $K_{\min}$, $K_{\max}$, and $R_0$), $C_\ast > 0$ (only depending on $\Gamma$, $g$, $K_{\min}$, $K_{\max}$, $C_0$, $\gamma_0$, $d$, and $p$), and $\gamma_\ast > 0$ (only depending on $\Gamma$, $g$, $K_{\min}$, $K_{\max}$, and $\gamma_0$) such that, for every $t > T_\ast$, we have
\[
\mathbf W_p(m_t, m_\infty)^p \leq C_\ast t^{p + d - 1} e^{-\gamma_\ast t}.
\]
\end{enumerate}
\end{corollary}

\begin{proof}
We first note that, for $X = \mathbb R^d$ endowed with the Euclidean distance, \ref{Hypo-X-SigmaCompact} is satisfied (and we take $\mathbf 0$ as the origin $0$ of $\mathbb R^d$ in the sequel) and \ref{Hypo-X-dist} is satisfied with $D = 1$ (by taking $\gamma$ to be the straight line connecting $x$ to $y$ with unit speed). Moreover, in both \ref{item:poly-decay} and \ref{item:exp-decay}, we have $m_0 \in \mathcal P_p(\mathbb R^d)$. Indeed, this is immediate in \ref{item:exp-decay}, while, in \ref{item:poly-decay}, it is a consequence of the assumption that $\beta \in (p + d, +\infty)$.

Take $R_0 > 0$ and $C_0 > 0$ as in \ref{item:poly-decay} or \ref{item:exp-decay}. Let $\psi$ be as in Theorem~\ref{ThmAsymp} and note that, by Proposition~\ref{prop-T-psi} and Corollary~\ref{coro-k-from-K}, $\psi$ satisfies $\psi(R) \leq A R + B$ for every $R > 0$, for some nonnegative constants $A$ and $B$ only depending on $\Gamma$, $g$, $K_{\min}$, and $K_{\max}$. In particular, $\psi(R) \leq A_\ast R$ for every $R \geq 1$, where $A_\ast = A + B$. Let $\alpha > 0$ and $t_0 \geq 0$ be as in Theorem~\ref{ThmAsymp}\ref{ThmAsymp-Wp} and set $T_0 = t_0 + \frac{1}{\alpha} \max(R_0, 1)$. Hence, for every $t > T_0$, we have $\alpha(t - t_0) > \max(R_0, 1)$, and thus, by Theorem~\ref{ThmAsymp}\ref{ThmAsymp-Wp}, we have
\begin{equation}
\label{eq:poly-exp-common-estim}
\mathbf W_p(m_t, m_\infty)^p \leq (2 A_\ast)^p \int_{\mathbb R^d \setminus \overline B_{\mathbb R^d}(0, \alpha(t - t_0))} \abs{x}^p \diff m_0(x).
\end{equation}
In the sequel, we denote by $\omega_d$ the surface of the unit sphere in $\mathbb R^d$.

Let us prove \ref{item:poly-decay}. Let $T_\ast = \max(T_0, 2 t_0)$. For every $t > T_\ast$, we have from \eqref{eq:poly-exp-common-estim} that
\begin{align*}
\mathbf W_p(m_t, m_\infty)^p & \leq (2 A_\ast)^p C_0 \int_{\mathbb R^d \setminus \overline B_{\mathbb R^d}(0, \alpha(t - t_0))} \frac{1}{\abs{x}^{\beta - p}}\diff x \\
& = (2 A_\ast)^p C_0 \omega_d \int_{\alpha(t - t_0)}^{+\infty} \frac{1}{r^{\beta - p - d + 1}}\diff r \\
& = \frac{(2 A_\ast)^p C_0 \omega_d}{(\beta - p - d) \alpha^{\beta - p - d} (t - t_0)^{\beta - p - d}} \leq \frac{2^{\beta - d} A_\ast^p C_0 \omega_d}{(\beta - p - d) \alpha^{\beta - p - d} t^{\beta - p - d}},
\end{align*}
where we use in the last inequality that $t - t_0 > \frac{t}{2}$ for $t > T_\ast$. We then get the desired conclusion with $C_\ast = \frac{2^{\beta - d} A_\ast^p C_0 \omega_d}{(\beta - p - d) \alpha^{\beta - p - d}}$.

Let us now prove \ref{item:exp-decay}. Let $T_\ast = T_0$. For every $t > T_\ast$, we have from \eqref{eq:poly-exp-common-estim} that
\begin{align*}
\mathbf W_p(m_t, m_\infty)^p & \leq (2 A_\ast)^p C_0 \int_{\mathbb R^d \setminus \overline B_{\mathbb R^d}(0, \alpha(t - t_0))} \abs{x}^p e^{-\gamma_0 \abs{x}} \diff x \\
& = (2 A_\ast)^p C_0 \omega_d \int_{\alpha(t - t_0)}^{+\infty} r^{p+d-1} e^{-\gamma_0 r} \diff r.
\end{align*}
Performing a change of variables in the integral, we can rewrite the last term of the above equation as
\[
(2 A_\ast)^p C_0 \omega_d (\alpha(t - t_0))^{p+d-1} e^{-\gamma_0 \alpha (t - t_0)} \int_{0}^{+\infty} \left(1 + \frac{r}{\alpha(t - t_0)}\right)^{p+d-1} e^{-\gamma_0 r} \diff r
\]
and, recalling that $1 < \alpha(t - t_0) \leq \alpha t$, we deduce that
\[
\mathbf W_p(m_t, m_\infty)^p \leq (2 A_\ast)^p C_0 \omega_d (\alpha t)^{p+d-1} e^{-\gamma_0 \alpha (t - t_0)} \int_{0}^{+\infty} \left(1 + r\right)^{p+d-1} e^{-\gamma_0 r} \diff r,
\]
which is the desired conclusion with $\gamma_\ast = \gamma_0 \alpha$ and 
\[C_\ast = (2 A_\ast)^p C_0 \omega_d \alpha^{p+d-1} e^{\gamma_0 \alpha t_0} \int_{0}^{+\infty} \left(1 + r\right)^{p+d-1} e^{-\gamma_0 r} \diff r. \qedhere\]
\end{proof}

\subsection{Dependence on the initial distribution of agents}
\label{sec:dependence-m0}

In Section~\ref{sec:existence}, we have established that, under our standing assumptions \ref{Hypo-X-SigmaCompact}--\ref{Hypo-X-dist} and \ref{HypoMFG-K-Bound}--\ref{HypoMFG-g-compatible}, for any $m_0 \in \mathcal P(X)$, the mean field game $\MFG(X, \Gamma, K, g, m_0)$ admits an equilibrium $Q \in \mathcal P(\mathcal C_X)$ and, in Section~\ref{sec:asymptotic}, we have considered the asymptotic behavior of such equilibria, i.e., the behavior of $m_t = {e_t}_{\#} Q$ as $t \to +\infty$, showing that these measures converge to the limit measure $m_\infty \in \mathcal P(X)$ characterized by $m_\infty = {e_\infty}_\# Q$. Our next goal is to understand how equilibria $Q$ and limit measures $m_\infty$ depend on $m_0$. Since both $Q$ and $m_\infty$ are not uniquely determined by $m_0$, we will make use a set-valued framework.

More precisely, given a metric space $(X, \dist)$, a nonempty closed subset $\Gamma$ of $X$, and functions $K\colon \mathcal P(X) \times X \to \mathbb R_+$ and $g\colon \Gamma \to \mathbb R_+$, we introduce the set-valued map $\Eq\colon \mathcal P(X) \rightrightarrows \mathcal P(\mathcal C_X)$ defined, for $m_0 \in \mathcal P(X)$, by
\begin{equation}
\label{eq:def-eq}
\Eq(m_0) = \{Q \in \mathcal P(\mathcal C_X) \suchthat Q \text{ is an equilibrium of } \MFG(X, \Gamma, K, g, m_0)\}.
\end{equation}
Note that $\Eq(m_0)$ can be equivalently rewritten as
\[
\Eq(m_0) = \{Q \in \mathcal P(\mathcal C_X) \suchthat {e_0}_\# Q = m_0 \text{ and } Q(\OOpt(Q)) = 1\}.
\]
We also introduce the set-valued map $\Lim\colon \mathcal P(X) \rightrightarrows \mathcal P(X)$ by setting, for $m_0 \in \mathcal P(X)$,
\begin{equation}
\label{eq:def-Lim}
\Lim(m_0) = \{{e_\infty}_\# Q \suchthat Q \in \Eq(m_0)\}.
\end{equation}
Our main result on the set-valued map $\Eq$ is the following one.

\begin{theorem}
\label{thm:eq-usc}
Consider the mean field game $\MFG(X, \Gamma, K, g, m_0)$, assume that \ref{Hypo-X-SigmaCompact}--\ref{Hypo-X-dist} and \ref{HypoMFG-K-Bound}--\ref{HypoMFG-g-compatible} are satisfied, and let $\Eq\colon \mathcal P(X) \rightrightarrows \mathcal P(\mathcal C_X)$ be the set-valued map defined in \eqref{eq:def-eq}. Then $\Eq(m_0)$ is nonempty and compact for every $m_0 \in \mathcal P(X)$ and $\Eq$ is upper semicontinuous. In particular, $\Eq$ has closed graph.
\end{theorem}

\begin{proof}
The fact that $\Eq(m_0)$ is nonempty for every $m_0 \in \mathcal P(X)$ is an immediate consequence of Theorem~\ref{thm-exist-equilibrium}.

We next show that $\Eq$ has closed graph. To do that, let $(m_{0, n})_{n \in \mathbb N}$ be a sequence in $\mathcal P(X)$ converging to some $m_0 \in \mathcal P(X)$ and $(Q_n)_{n \in \mathbb N}$ be a sequence in $\mathcal P(\mathcal C_X)$ converging to some $Q \in \mathcal P(\mathcal C_X)$ and with $Q_n \in \Eq(m_{0, n})$ for every $n \in \mathbb N$. Hence ${e_0}_\# Q_n = m_{0, n}$ and $Q_n(\OOpt(Q_n)) = 1$ for every $n \in \mathbb N$.

Since ${e_0}_\# Q_n = m_{0, n}$ for every $n \in \mathbb N$, it follows from Lemma~\ref{lemm-etQ-continuous} that ${e_0}_\# Q = m_0$. To prove that $Q (\OOpt(Q)) = 1$, let $\varepsilon \in (0, 1)$ and define the set $V_\varepsilon$ by $V_{\varepsilon} = \{\gamma \in \mathcal C_X \suchthat \dist_{\mathcal C_X}(\gamma, \OOpt(Q)) \leq \varepsilon\}$. Let $R_0 > 0$ be such that $m_0\bigl(B_X(\mathbf 0, R_0)\bigr) \geq 1 - \frac{\varepsilon}{2}$ and note that, by the Portmanteau theorem, we have $\liminf_{n \to +\infty} m_{0, n}\bigl(B_X(\mathbf 0, R_0)\bigr) \geq m_0\bigl(B_X(\mathbf 0, R_0)\bigr)$, hence there exists $N_\varepsilon \in \mathbb N$ such that, for all $n \in \mathbb N$ with $n \geq N_\varepsilon$, we have $m_{0, n}\bigl(B_X(\mathbf 0, R_0)\bigr) \geq 1 - \varepsilon$.

Proceeding exactly as in the proof of Lemma~\ref{lemm-F}, we deduce that, up to increasing $N_{\varepsilon}$, we have
\begin{equation*}
\Lip_{K_{\max}}\bigl(\overline B_X(\mathbf 0, \psi(R_0))\bigr) \cap \OOpt(Q_n) \subset V_\varepsilon
\end{equation*}
for all $n \in \mathbb N$ with $n \geq N_\varepsilon$, where $\psi$ is the function from Proposition~\ref{prop-equilibrium-psi}. Since $Q_n \in \Eq(m_{0, n})$, we obtain, by combining the above inclusion with Proposition~\ref{prop-equilibrium-psi} and the fact that $Q_n(\OOpt(Q_n)) = 1$, that, for every $n \in \mathbb N$ with $n \geq N_\varepsilon$, we have
\begin{align*}
Q_n(V_{\varepsilon}) & \geq Q_n\Bigl(\Lip_{K_{\max}}\bigl(\overline B_X(\mathbf 0, \psi(R_0))\bigr) \cap \OOpt(Q_n)\Bigr) \\
& \geq m_{0, n}\bigl(\overline B_X(\mathbf 0, R_0)\bigr) \geq m_{0, n}\bigl(B_X(\mathbf 0, R_0)\bigr) \geq 1 - \varepsilon.
\end{align*}
Proceeding once again as in the proof of Lemma~\ref{lemm-F}, we deduce that $Q(V_{\varepsilon}) \geq 1 - \varepsilon$ and thus $Q(\OOpt(Q)) = \lim_{\varepsilon \to 0} Q(V_{\varepsilon}) = 1$, completing the proof that $Q \in \Eq(m_0)$. Hence, $\Eq$ has closed graph.

Since $\Eq$ has closed graph, we deduce that $\Eq(m_0)$ is closed for every $m_0 \in \mathcal P(X)$. On the other hand, by Proposition~\ref{prop-equilibrium-psi}, $\Eq(m_0)$ is a subset of the set $\mathfrak Q$ defined in \eqref{eq-defi-Q}, which is compact thanks to Proposition~\ref{PropQNonemptyConvexCompact}, implying thus that $\Eq(m_0)$ is compact.

To conclude the proof, assume, to obtain a contradiction, that $\Eq$ is not upper semicontinuous. Hence there exist $m_0 \in \mathcal P(X)$, a neighborhood $U$ of $\Eq(m_0)$ in $\mathcal P(\mathcal C_X)$, a sequence $(m_{0, n})_{n \in \mathbb N}$ converging to $m_0$, and a sequence $(Q_n)_{n \in \mathbb N}$ in $\mathcal P(\mathcal C_X)$ such that $Q_n \in \Eq(m_{0, n}) \setminus U$ for every $n \in \mathbb N$.

We claim that the sequence $(Q_n)_{n \in \mathbb N}$ is tight. Indeed, let $\varepsilon \in (0, 1)$ and take $R_0 > 0$ such that $m_0\bigl(B_X(\mathbf 0, R_0)\bigr) \geq 1 - \frac{\varepsilon}{2}$. By the Portmanteau theorem, we have
\[\liminf_{n \to +\infty} m_{0, n}\bigl(B_X(\mathbf 0, R_0)\bigr) \geq m_0\bigl(B_X(\mathbf 0, R_0)\bigr) \geq 1 - \frac{\varepsilon}{2},\]
hence there exists $N_\varepsilon \in \mathbb N$ such that, for all $n \in \mathbb N$ with $n \geq N_\varepsilon$, we have $m_{0, n}\bigl(B_X(\mathbf 0, R_0)\bigr) \geq 1 - \varepsilon$. Since $Q_n \in \Eq(m_{0, n})$, we deduce from Proposition~\ref{prop-equilibrium-psi} that
\[Q_n\Bigl(\Lip_{K_{\max}}\bigl(\overline B_X(\mathbf 0, \psi(R_0))\bigr)\Bigr) \geq 1 - \varepsilon\]
for every $n \in \mathbb N$ with $n \geq N_\varepsilon$, and thus the tightness of $(Q_n)_{n \in \mathbb N}$ follows as a consequence of the compactness of $\Lip_{K_{\max}}\bigl(\overline B_X(\mathbf 0, \psi(R_0))\bigr)$. Hence, by Prokhorov's theorem (see, e.g., \cite[Theorem~5.1.3]{Ambrosio2005Gradient}), up to extracting a subsequence, which we still denote by $(Q_n)_{n \in \mathbb N}$ for simplicity, there exists $Q \in \mathcal P(\mathcal C_X)$ such that $Q_n \to Q$ as $n \to +\infty$. Since $\Eq$ has closed graph, we conclude that $Q \in \Eq(m_0)$. This, however, contradicts the fact that $Q_n \notin U$ for every $n \in \mathbb N$, yielding the desired conclusion.
\end{proof}

Note that the set-valued map $\Lim$ can be seen as the composition of the set-valued map $\Eq\colon \mathcal P(X) \rightrightarrows \mathcal P(\mathcal C_X)$ with the operation of pushforward by $e_\infty$. The main difficulty in proving a result similar to Theorem~\ref{thm:eq-usc} for the set-valued map $\Lim$ is that $e_\infty$ is not defined on the whole space $\mathcal C_X$ and, even in the subset where it is defined, it is not continuous. We will prove, however, that the operation of pushforward by $e_\infty$ is continuous on the range of $\Eq$. For that purpose, let us first introduce the set
\begin{equation}
\label{eq:def-eeq}
\EEq = \bigcup_{m_0 \in \mathcal P(X)} \Eq(m_0),
\end{equation}
which is the range of $\Eq$. Our first preliminary result, obtained as a simple consequence of Theorem~\ref{thm:eq-usc}, is that this set is closed.

\begin{lemma}
Consider the mean field game $\MFG(X, \Gamma, K, g, m_0)$ and assume that \ref{Hypo-X-SigmaCompact}--\ref{Hypo-X-dist} and \ref{HypoMFG-K-Bound}--\ref{HypoMFG-g-compatible} are satisfied. Then the set $\EEq$ defined in \eqref{eq:def-eeq} is closed.
\end{lemma}

\begin{proof}
Let $(Q_n)_{n \in \mathbb N}$ be a sequence in $\EEq$ converging to some $Q$ in $\mathcal P(\mathcal C_X)$. Clearly, $Q_n \in \Eq({e_0}_\# Q_n)$ and, by Lemma~\ref{lemm-etQ-continuous}, we have ${e_0}_\# Q_n \to {e_0}_\# Q$ as $n \to +\infty$. Hence, by Theorem~\ref{thm:eq-usc}, we deduce that $Q \in \Eq({e_0}_\# Q)$, showing that $Q \in \EEq$, as required.
\end{proof}

We shall also need the following preliminary result.

\begin{lemma}
\label{lemm:Q-constant-large}
Consider the mean field game $\MFG(X, \Gamma, K, g, m_0)$, assume that \ref{Hypo-X-SigmaCompact}--\ref{Hypo-X-dist} and \ref{HypoMFG-K-Bound}--\ref{HypoMFG-g-compatible} are satisfied, and let $\EEq$ be the set defined in \eqref{eq:def-eeq}. Let $Q \in \EEq$ and denote $m_0 = {e_0}_\# Q$. Then, for every $R > 0$, we have
\[Q\bigl(\{\gamma \in \mathcal C_X \suchthat \gamma \text{ is constant on } [T(R), +\infty)\}\bigr) \geq m_0\bigl(\overline B_X(\mathbf 0, R)\bigr),\]
where $T$ is the function from Proposition~\ref{prop-T-psi}.
\end{lemma}

\begin{remark}
\label{remk:T}
The function $T$ from Proposition~\ref{prop-T-psi} is associated with an optimal control problem $\OCP(X, \Gamma, k, g)$ satisfying \ref{Hypo-X-SigmaCompact}--\ref{HypoOCP-k-Bound} and \ref{HypoOCP-g-compatible} and it depends only on $\mathbf 0$, $\Gamma$, $g$, $D$, $K_{\min}$, and $K_{\max}$. In the context of Lemma~\ref{lemm:Q-constant-large}, instead of an optimal control problem, we consider a mean field game $\MFG(X, \Gamma, K, g, m_0)$. The function $T$ in its statement should be understood as the function $T$ associated with the optimal control problem $\OCP(X, \Gamma, k_Q, g)$, where $k_Q$ is defined from $K$ and $Q$ as in Corollary~\ref{coro-k-from-K}, and it follows from the latter result that $T$ depends only on $\mathbf 0$, $\Gamma$, $g$, $D$, $K_{\min}$, and $K_{\max}$ from \ref{Hypo-X-SigmaCompact}--\ref{Hypo-X-dist} and \ref{HypoMFG-K-Bound}--\ref{HypoMFG-g-compatible}.
\end{remark}

\begin{proof}[Proof of Lemma~\ref{lemm:Q-constant-large}]
Let $R > 0$ and take $\gamma \in \OOpt(Q)$ such that $\gamma(0) \in \overline B_X(\mathbf 0, R)$. It follows from Definition~\ref{DefiOCP} and Proposition~\ref{prop-T-psi} that $\tau(0, \gamma) \leq T(R)$ and that $\gamma$ is constant on $[\tau(0, \gamma), +\infty)$. Hence, for every $R > 0$, we have
\[
\{\gamma \in \mathcal C_X \suchthat \gamma \in \OOpt(Q) \text{ and } \gamma(0) \in \overline B_X(\mathbf 0, R)\} \subset \{\gamma \in \mathcal C_X \suchthat \gamma \text{ is constant on } [T(R), +\infty)\},
\]
showing that
\begin{multline*}
Q\bigl(\{\gamma \in \mathcal C_X \suchthat \gamma \text{ is constant on } [T(R), +\infty)\}\bigr) \\ \geq Q\bigl(\{\gamma \in \mathcal C_X \suchthat \gamma \in \OOpt(Q) \text{ and } \gamma(0) \in \overline B_X(\mathbf 0, R)\}\bigr).
\end{multline*}
Since $Q \in \EEq$, we have $Q(\OOpt(Q)) = 1$, thus
\begin{align*}
Q\bigl(\{\gamma \in \mathcal C_X \suchthat \gamma \in \OOpt(Q) \text{ and } \gamma(0) \in \overline B_X(\mathbf 0, R)\}\bigr) & = Q\bigl(\{\gamma \in \mathcal C_X \suchthat \gamma(0) \in \overline B_X(\mathbf 0, R)\}\bigr) \\
& = m_0\bigl(\overline B_X(\mathbf 0, R)\bigr),
\end{align*}
yielding the conclusion.
\end{proof}

We are now in position to show that the operation of pushforward by $e_\infty$ is continuous in the range of $\Eq$.

\begin{lemma}
\label{lemm:einftysharp-is-continuous}
Consider the mean field game $\MFG(X, \Gamma, K, g, m_0)$, assume that \ref{Hypo-X-SigmaCompact}--\ref{Hypo-X-dist} and \ref{HypoMFG-K-Bound}--\ref{HypoMFG-g-compatible} are satisfied, and let $\EEq$ be the set defined in \eqref{eq:def-eeq}. Then the map $\EEq \ni Q \mapsto {e_\infty}_\# Q \in \mathcal P(X)$ is continuous.
\end{lemma}

\begin{proof}
Let $(Q_n)_{n \in \mathbb N}$ be a sequence in $\EEq$ converging to some $Q \in \EEq$. Define $m_0 = {e_0}_\# Q$ and, for $n \in \mathbb N$, set $m_{0, n} = {e_0}_\# Q_n$. Note that $m_{0, n} \to m_0$ in $\mathcal P(X)$ as $n \to +\infty$ thanks to Lemma~\ref{lemm-etQ-continuous}.

Take $\varepsilon > 0$ and let $R_\varepsilon > 0$ be such that $m_0\bigl(B_X(\mathbf 0, R_\varepsilon)\bigr) \geq 1 - \frac{\varepsilon}{4}$. Using the fact that $m_{0, n} \to m_0$ in $\mathcal P(X)$ as $n \to +\infty$ and the Portmanteau theorem, we deduce that there exists $N_\varepsilon \in \mathbb N$ such that, for every $n \in \mathbb N$ with $n \geq N_\varepsilon$, we have
\[
m_{0, n}\bigl(B_X(\mathbf 0, R_\varepsilon)\bigr) \geq 1 - \frac{\varepsilon}{2}.
\]
Let $T$ be the function defined in Proposition~\ref{prop-T-psi} (considered here in the sense of Remark~\ref{remk:T}) and set $T_\varepsilon = T(R_\varepsilon)$ and
\[C_\varepsilon = \{\gamma \in \mathcal C_X \suchthat \gamma \text{ is constant on } [T_\varepsilon, +\infty)\}.\]
Note that $C_\varepsilon$ is closed and, from Lemma~\ref{lemm:Q-constant-large}, we have
\begin{equation}
\label{eq:estim-Qn-constant}
Q_n(C_\varepsilon) \geq 1 - \frac{\varepsilon}{2} \text{ for every } n \in \mathbb N \text{ with } n \geq N_\varepsilon.
\end{equation}

Let $F \subset X$ be a closed set. For every $n \in \mathbb N$, we have
\begin{equation}
\label{eq:decompose-einftyQnF}
{e_\infty}_\# Q_n (F) = Q_n(e_\infty^{-1}(F)) = Q_n(e_\infty^{-1}(F) \cap C_\varepsilon) + Q_n(e_\infty^{-1}(F) \setminus C_\varepsilon).
\end{equation}
Note that, since any trajectory $\gamma \in C_\varepsilon$ is constant on $[T_\varepsilon, +\infty)$, we have $e_\infty^{-1}(F) \cap C_\varepsilon = e_{T_\varepsilon}^{-1}(F) \cap C_\varepsilon$. Since $e_{T_\varepsilon}$ is continuous by Lemma~\ref{lemm-etQ-continuous} and $F$ is closed, we have that $e_{T_\varepsilon}^{-1}(F)$ is closed, and thus so is $e_{T_\varepsilon}^{-1}(F) \cap C_\varepsilon$. Hence, by the Portmanteau theorem, up to increasing $N_\varepsilon$ (in a way that also depends on $F$), we have that $Q_n(e_\infty^{-1}(F) \cap C_\varepsilon) \leq Q(e_\infty^{-1}(F) \cap C_\varepsilon) + \frac{\varepsilon}{2}$. Combining this with \eqref{eq:estim-Qn-constant} and \eqref{eq:decompose-einftyQnF}, we deduce that, for $n \geq N_\varepsilon$, we have
\[
{e_\infty}_\# Q_n (F) \leq Q(e_\infty^{-1}(F) \cap C_\varepsilon) + \varepsilon \leq Q(e_\infty^{-1}(F)) + \varepsilon = {e_\infty}_\# Q(F) + \varepsilon,
\]
which shows that $\limsup_{n \to +\infty} {e_\infty}_\# Q_n (F) \leq {e_\infty}_\# Q(F)$. Since this holds for any closed set $F \subset X$, it follows from the Portmanteau theorem that ${e_\infty}_\# Q_n \to {e_\infty}_\# Q$ as $n \to +\infty$, as required.
\end{proof}

We can now deduce, as an immediate consequence of Theorem~\ref{thm:eq-usc} and Lemma~\ref{lemm:einftysharp-is-continuous}, that the counterpart of the former also holds for the set-valued $\Lim$. More precisely, we have the following result.

\begin{corollary}
Consider the mean field game $\MFG(X, \Gamma, K, g, m_0)$, assume that \ref{Hypo-X-SigmaCompact}--\ref{Hypo-X-dist} and \ref{HypoMFG-K-Bound}--\ref{HypoMFG-g-compatible} are satisfied, and let $\Lim\colon \mathcal P(X) \rightrightarrows \mathcal P(X)$ be the set-valued map defined in \eqref{eq:def-Lim}. Then $\Lim(m_0)$ is nonempty and compact for every $m_0 \in \mathcal P(X)$ and $\Lim$ is upper semicontinuous. In particular, $\Lim$ has closed graph.
\end{corollary}

\begin{remark}
In addition to upper semicontinuity, there are several other notions of continuity for set-valued maps, such as lower semicontinuity, Lipschitz continuity, or also continuity with respect to the Hausdorff distance (see, e.g., \cite{Aubin2009Set}). In general, however, no further continuity properties should be expected for the maps $\Eq$ and $\Lim$. Consider, for instance, a modification of the mean field game from Remark~\ref{RemkEquilibriumNonUnique} in which the initial distribution of agents $m_0$ is replaced by the measure $m_{0, a} = \delta_a$, where $a \in [0, 1]$ and $\delta_a$ denotes the Dirac delta measure at $a$. One can then compute that, for this mean field game,
\begin{align*}
\Eq(\delta_a) & = \begin{dcases*}
\{\delta_{\gamma_{\ell, a}}\} & if $0 \leq a < \frac{1}{2}$, \\
\{\alpha \delta_{\gamma_{\ell, a}} + (1 - \alpha) \delta_{\gamma_{r, a}} \suchthat \alpha \in [0, 1]\} & if $a = \frac{1}{2}$, \\
\{\delta_{r, a}\} & if $\frac{1}{2} < a \leq 1$,
\end{dcases*} \displaybreak[0] \\
\Lim(\delta_a) & = \begin{dcases*}
\{\delta_{0}\} & if $0 \leq a < \frac{1}{2}$, \\
\{\alpha \delta_{0} + (1 - \alpha) \delta_{1} \suchthat \alpha \in [0, 1]\} & if $a = \frac{1}{2}$, \\
\{\delta_{1}\} & if $\frac{1}{2} < a \leq 1$,
\end{dcases*}
\end{align*}
where $\gamma_{\ell, a}, \gamma_{r, a} \in \mathcal C_X$ are the trajectories defined for $t \in \mathbb R_+$ by $\gamma_{\ell, a}(t) = \max\left(a - t, 0\right)$ and $\gamma_{r, a}(t) = \min\left(a + t, 1\right)$. It is immediate to verify from the above expressions that $\Eq$ and $\Lim$ are not lower semicontinuous, Lipschitz continuous, nor continuous with respect to the Hausdorff distance at $\delta_{\frac{1}{2}}$.
\end{remark}

\section*{Acknowledgment}

The author wishes to thank the anonymous reviewers of this paper, whose remarks led to important improvements. In particular, Section~\ref{sec:dependence-m0} was written following a question by a reviewer on the stability of $m_\infty$ with respect to $m_0$.

\bibliographystyle{abbrv}
\bibliography{Bibliography}

\end{document}